\documentclass[a4paper,11pt]{amsart}

\usepackage{amsmath, amssymb, comment}
\usepackage{mathrsfs,wasysym}
\usepackage{graphicx}
\usepackage{hyperref}
\usepackage{wrapfig}
\usepackage{enumerate}
\usepackage{setspace}

\usepackage{color}

\newcommand{\bpi}{\boldsymbol\pi}
\newcommand{\pgf}{\boldsymbol p}
\newcommand{\llgf}{\bar{\boldsymbol p}}
\newcommand{\mgf}{\tilde{\boldsymbol p}}
\newcommand{\od}[2]{\frac{{\rm d} #1}{{\rm d} #2}}
\newcommand{\pd}[2]{\frac{\partial #1}{\partial #2}}
\newcommand{\m}{\boldsymbol}
\newcommand{\M}{{{\mathcal{M}}}}

\renewcommand{\vec}[1]{\boldsymbol{#1}}

\newcommand{\ee}{\mathbb E}

\newtheorem{thm}{Theorem}
\newtheorem{lem}{Lemma}
\newtheorem{cor}{Corollary}
\newtheorem{prop}{Proposition}

\newtheorem{rem}{Remark}

\newcommand{\hI} {^{{\scriptsize(\mbox{\sc i} )}}}
\newcommand{\hII}{^{{\scriptsize(\mbox{\sc ii} )}}}
\newcommand{\afg}{\frac{{\rm d}\varrho_k\hII(t)}{{\rm d}t}}

\newcommand{\vt}{\vartheta}
\newcommand{\vr}{\varrho}

\newcommand{\vb}{\vspace{3.2mm}}

\title%
[Markov-modulated infinite-server queues]
{Analysis of {Markov-modulated infinite-server queues \\
in the central-limit regime}}

\author{Joke Blom$\,^\star$, Koen De Turck$\,^\dagger$, Michel Mandjes$\,^{\bullet,\star}$}

\date{\today}

\begin{document}
\maketitle

\begin{abstract}
This paper focuses on an infinite-server queue modulated by an independently evolving 
finite-state Markovian
background process, with transition rate matrix $Q\equiv(q_{ij})_{i,j=1}^d$. 
{Both arrival rates and service rates are depending
on the state of the background process.}
The main contribution concerns the derivation of central limit theorems for the number of customers
in the system at time $t\ge 0$, in the asymptotic regime
in which 
 the arrival rates $\lambda_i$ are scaled by a factor $N$, and the transition  rates $q_{ij}$
by a factor $N^\alpha$, with $\alpha \in \mathbb R^+$. 
The specific value of $\alpha$ has a crucial impact on the result:
(i)~for $\alpha>1$ the system essentially behaves as an M/M/$\infty$ queue, 
and in the central limit theorem the centered process has to be normalized by
$\sqrt{N}$; (ii)~for $\alpha<1$,
the centered process has to be normalized by
$N^{{1-}\alpha/2}$, with the deviation matrix 
appearing in the expression for the variance.

\vspace{3mm}

\vspace{2mm}

\vspace{2mm}

\noindent {Work done while K.\ de Turck was visiting Korteweg-de Vries Institute for Mathematics,
University of Amsterdam, the Netherlands, with greatly appreciated financial support from {\it Fonds Wetenschappelijk Onderzoek / Research Foundation -- Flanders}.
He is also a Postdoctoral Fellow of the same foundation.

\begin{itemize}
\item[$^\bullet$] Korteweg-de Vries Institute for Mathematics,
University of Amsterdam, Science Park 904, 1098 XH Amsterdam, the Netherlands.
\item[$^\star$] CWI, P.O. Box 94079, 1090 GB Amsterdam, the Netherlands.
\item[$^\dagger$] TELIN, Ghent University, St.-Pietersnieuwstraat 41,
B9000 Gent, Belgium.
\end{itemize}
\noindent
M.\ Mandjes is also with  E{\sc urandom}, Eindhoven University of Technology, Eindhoven, the Netherlands,
and IBIS, Faculty of Economics and Business, University of Amsterdam,
Amsterdam, the Netherlands.

\noindent
{\tt joke.blom@cwi.nl, kdeturck@telin.ugent.be, M.R.H.Mandjes@uva.nl}}

\end{abstract}

\newcommand{\hN}{^{(N)}}
\newcommand{\MEAN}{{\mathbb E}}
\newcommand{\PROB}{{\mathbb P}}
\newcommand{\VAR}{{\mathbb V}{\rm ar}}
\newcommand{\COV}{{\mathbb C}{\rm ov}}

\newpage

\section{Introduction} 

The infinite-server queue has been intensively
studied, perhaps owing to its wide applicability and attractive
computational features.
In these
systems jobs arrive according to a given arrival process, go into
service immediately, are served in parallel, and leave when their service is
completed. An important feature of this model is that there is {\it no
waiting}:  jobs do not interfere with each other.
The infinite-server queue was originally developed to analyze 
the probabilistic properties of
the number of calls in progress in a trunk in a communication network,
as an approximation of the corresponding system with many servers.
More recently, however, 
various other application domains have been identified, such as road traffic \cite{WOEN} and
biology \cite{BRUGGEMAN2}.

In the most standard variant of the infinite-server model, 
known as the M/M/$\infty$ model, jobs arrive according to a Poisson
process with a fixed rate $\lambda$, where the service times are i.i.d.\
samples from an exponential distribution with mean $\mu^{-1}$ (independent of the job
arrival process). 
A classical result states that 
the stationary number of jobs in the
system has a Poisson distribution with mean $\lambda/\mu$.
Also the transient behavior of this queueing system is well understood.

\vb
In many
practical situations, however, the assumptions underlying the standard
infinite-server model are not valid.
The arrivals often tend to be `clustered' (so that the assumption of a fixed arrival rate
does not apply), while also the service distribution may vary over time.
This explains the interest in {\it
Markov-modulated} infinite-server queues, so as to 
incorporate `burstiness' into the queue's input process.
In such queues, the input process is modulated by a finite-state (of dimension
$d\in{\mathbb N}$) irreducible continuous-time Markov process $(J(t))_{t\in{\mathbb R}}$,
often referred to as the {\it background process} or {\it modulating process}, with 
transition rate matrix $Q\equiv(q_{ij})_{i,j=1}^d$. If $J(t)$ is
in state $i$, the arrival process is (locally) a Poisson process with rate
$\lambda_i$ {and} the service times are exponential with mean $\mu_i^{-1}$
(while the obvious independence assumptions are assumed to be fulfilled).

The Markov-modulated infinite-server queue has attracted some
attention over the past decades
(but the number of papers on this type of system is relatively modest, compared to the
vast literature on Markov-modulated single-server queues). 
The main focus in the literature so far has been on
characterizing the steady-state number of jobs in the system; see e.g.\
\cite{DAURIA,FALIN2008,FRALIXADAN2009,KEILSONSERVI1993,OCINNEIDEPURDUE}
and references therein. Interestingly, there are hardly any explicit results on the 
probability distribution of the (transient or stationary) number of jobs present: 
the results are in terms of recursive schemes to determine all moments, and
implicit characterizations of the probability
generating function.

An idea to obtain more explicit results for the distribution 
of the number of jobs in the system, 
is by applying specific {\it time-scalings}.
In \cite{BKMT,HELL} a time-scaling is studied in which the
transitions of the background process occur at a faster rate than the Poisson
arrivals. As a consequence, the limiting input process becomes
essentially Poisson (with an arrival rate being an average of the $\lambda_i$\,s);
a similar property applies for the service times. Under this scaling,
one gets in the limit  the Poisson distribution for the 
stationary {number} of
jobs present. Recently, related transient results have been obtained as well,
under specific scalings of the arrival rates and transition times of the
background process \cite{BKMT,BMT}.

\vb

{\it Contribution.} Our work considers a time-scaling featuring in
\cite{BKMT,BMT} as well. In this scaling, the arrival rates $\lambda_i$ are inflated by a factor $N$, while the 
background process $(J(t))_{t\in{\mathbb R}}$ is sped up by a factor $N^\alpha$, for some $\alpha \in (0,\infty)$.
The primary focus is on the regime in which $N$ grows large.

The object of study is the number of jobs in the scaled system at time $t$, in the sequel denoted by $M^{(N)}(t)$.
More specifically, we aim at deriving a central limit theorem ({\sc clt}) for $M^{(N)}(t)$, as well as for its stationary counterpart $M^{(N)}$.
Interestingly, we find different
scaling regimes, based on the value of $\alpha$. The rationale behind these 
different regimes lies in the fact that for $\alpha>1$ the variances of $M^{(N)}(t)$
and $M^{(N)}$ 
grow essentially linearly in $N$, while for $\alpha<1$ they grow as $N^{2-\alpha}.$

It is important to notice that there are actually two variants of this Markov-modulated infinite-server queue. 
In the first (to be referred to as `{Model {\sc i}}') the service times of jobs present at time $t$ are subject to a hazard rate that is
 determined by the state $J(t)$ of the background process at time $t$. In the second variant (referred to as `{Model {\sc ii}}') the service times 
 are determined by the state of the modulating process at the job's arrival epoch
 (and hence can be sampled upon arrival).

The main contribution of our work
is that we develop a unified approach
to prove the {\sc clt}\,s for both Model {\sc i} and Model {\sc ii} for the scalings given above, for arbitrary $\alpha\in(0,\infty)$, and for both the transient and stationary regimes.
The technique used can be summarized as follows.
We first derive differential equations for the probability generating functions (pgf\,s) of the transient number of jobs in the system $M^{(N)}(t)$ as well as
its stationary counterpart $M^{(N)}$ (for both models). 
The next step is to establish laws of large numbers: we identify $\varrho(t)$ ($\varrho$, respectively) to 
which $N^{-1} M^{(N)}(t)$ ($N^{-1}\,M^{(N)}$, respectively) converges as $N\to\infty$. This result indicates
how $M^{(N)}(t)$ and $M^{(N)}$
 should be  centered in a {\sc clt}. The thus obtained centered random variables are then normalized
 (that is, divided by $N^\gamma$, for an appropriately chosen $\gamma$), so as to obtain the {\sc clt}.
As suggested by the asymptotic behavior of the variance of $M^{(N)}(t)$ and $M^{(N)}$, as we pointed out above,
the appropriate choice of the parameter $\gamma$ in the normalization 
is $\gamma=\frac{1}{2}$ for $\alpha>1$, and $\gamma={1-\frac{\alpha}{2}}$ for $\alpha<1$.
The proofs rely on (non-trivial) manipulations of the differential equations that underly 
the pgf\,s. For $\alpha<1$ the {\it deviation matrix}  \cite{CS}
appears  in the {\sc clt} in the expression for the variance.

\vb

{\it Relation to previous work.}
In our preliminary conference paper \cite{BdTM} we just covered Model~{\sc i}, with an approach similar to the one
featuring in the present paper.
In \cite{BKMT} the transient regime of Model {\sc ii} is analyzed, but just for $\alpha>1$, relying on a  different and more elaborate methodology.
New results of this paper are: (i)~Model {\sc ii} for $\alpha \le 1$, (ii)~the {\sc clt} for the stationary number of jobs $M^{(N)}$ in Model {\sc ii}, (iii)
results on the correlation across time. The main contribution, however, concerns the unified approach: where earlier work has been using ad hoc solutions for the scenario at hand, we now have a general
`recipe' to derive {\sc clt}\,s of this kind. Current work in progress aims at functional versions of the {\sc clt}\,s for the {\it process} $(M^{(N)}(t))_{t\in{\mathbb R}}$;
\cite{DAVE} covers the special case of uniform
service rates, which constitutes the intersection of Model {\sc i} and Model {\sc ii}.

\vb

{\it Organization.} The organization of this paper is as follows.
In Section~\ref{sec:model}, we explain the model in detail and introduce the
notations used throughout the paper.
Section \ref{sec:Proof} provides a systematic explanation of our technique for proving this kind of {\sc clt}\,s;
in addition, we demonstrate the approach for a special case, viz.\ the transient analysis for the 
model with uniform  service rates (in which Model {\sc i} and Model {\sc ii} coincide).
In Section \ref{sec:hazard-diff-eq} we recall the results for Model {\sc i} as derived in the precursor paper \cite{BdTM}.
Then in Section \ref{sec:Mod2}, we state and prove for Model {\sc ii} the {\sc clt}\,s, both for the stationary and transient distribution.
The single-dimensional convergence
can be extended to convergence of the finite-dimensional distributions (viz.\ at different points in time); see
Section \ref{CAT}.
In Section \ref{sec:illustration}, we provide some numerical examples so as to get insight into the speed
of convergence to the various limiting regimes.
The final section of the paper,  Section~\ref{sec:discussion}, contains
a discussion and concluding remarks.

\section{Model description and preliminaries} 
\label{sec:model}
In this section, we first provide a detailed model description.
We then give a number of explicit calculations
for the mean and variance of $M^{(N)}(t)$, that indicate how this random variable should be centered and normalized so as to obtain a {\sc clt}. 
We conclude by presenting a number of preliminary results
(e.g., a number of standard results on deviation matrices).

\subsection{Model description, scaling}
The main objective of this paper is to study an infinite-server queue with
Markov-modulated Poisson arrivals and exponential  service times.
In full detail, the model is described as follows.

\vb

{\it Model.} Consider an irreducible continuous-time Markov process $(J(t))_{t\in {\mathbb
R}}$ on a finite state space $\{1,\ldots,d\}$, with $d\in{\mathbb N}$. Let its transition rate
matrix be given by $Q\equiv \left(q_{ij}\right)_{i,j=1}^d$; here the rates 
$q_{ij}$ are nonnegative if $i\not=j$, whereas $q_{ii}=-\sum_{j\not=i}q_{ij}$ (so that the row sums are $0$).
Let $\pi_i$ be the
stationary probability that the background process is in state $i$, for
$i=1,\ldots,d$; due to the irreducibility assumption there is a unique stationary distribution.  
The time spent in state $i$  (often referred to as the {\it
transition time}) has an exponential distribution with mean $1/q_i$, where
$q_i:=-q_{ii}$.

Let $M(t)$ denote the number of jobs in the system at time $t$, and $M$ its steady-state counterpart. The dynamics of
the process $(M(t))_{t\in{\mathbb R}}$ can be described as follows.
While the process $(J(t))_{t\in{\mathbb R}}$, usually referred to as the {\it
background process} or {\it modulating process}, is in state $i\in\{1,\ldots,d\}$, jobs arrive at the queue
according to a Poisson process with rate $\lambda_i\geq0$. The service times
are assumed to be exponentially distributed with rate $\mu_i$, however, more importantly
this statement can be interpreted in two ways: \vspace{2mm}
\begin{itemize}
\item[Model {\sc i}:]
In the first variant of our model, the service times of all jobs present at a certain time instant $t$ are subject to
a hazard rate determined by the state $J(t)$ of background chain at time $t$, regardless of
when they arrived. Informally, if the system is in state $i$, then the  probability of an arbitrary job leaving the system in the next $\Delta t$ time units is
$\mu_i\,\Delta t$.
\item[Model {\sc ii}:]
In the second variant the service rate is determined by the background state as seen by the job upon its arrival.
If the background process was in state $i$, the service time is sampled from an exponential distribution
with mean $\mu_i^{-1}.$
\end{itemize}
The difference between the two models is nicely illustrated by the following alternative representation \cite{DAURIA}.
In Model {\sc i} $M(t)$ has a Poisson distribution with random parameter $\psi(J)$, while in Model {\sc ii} it
is Poisson with random parameter $\varphi(J)$, where $J\equiv (J(s))_{s\in[0,t]}$, and
\begin{equation}\label{Pr}\psi(f) :=\int_0^t \lambda_{f(s) }e^{-\int_s^t \mu_{f(r )}{\rm d}r} {\rm d}s,\:\:\:\:\:\:\:
\varphi(f) :=\int_0^t \lambda_{f(s)} e^{-\mu_{f(s )}\,(t-s)} {\rm d}s,\end{equation}
with $f:[0,t]\mapsto \{1,\ldots,d\}.$

\vb

{\it Scaling.}
In this paper, we consider a scaling in which both (i) the arrival process, and (ii) the
background process are sped up, at a possibly distinct rate. More specifically, the arrival rates
are scaled linearly, that is, as $\lambda_i\mapsto N\lambda_i$, whereas the background chain is
scaled as $q_{ij}\mapsto N^\alpha q_{ij}$, for some {positive} $\alpha$.
We call the resulting 
process $(M\hN({t}))_{{t}\in {\mathbb R}}$, to stress the dependence on the scaling parameter $N$;
the corresponding background process is denoted by $(J\hN({t}))_{{t}\in {\mathbb R}}$.

The main objective of this paper is the derivation of {\sc clt}\,s for the number of jobs in the system, as $N$ grows large.
As mentioned in the introduction, the parameter $\alpha$ plays
an important role here: it turns out to matter
whether $\alpha$ is assumed smaller than, equal to, or larger than 1.
Letting the system start off empty at time $0$, we consider the number of jobs
present at time $t$, denoted by $M^{(N)}(t)$; we write $M^{(N)}$ for its stationary counterpart.

Our main result is a `non-standard {\sc clt}':
for a deterministic function $\vr(t)$,
\begin{equation}\label{eq:GenCLT}
\frac{M^{(N)}(t)-N{\varrho}(t)}{N^{{\gamma}}}
\end{equation}
converges in distribution to a zero-mean Normal distribution with a certain variance, say, $\sigma^2(t)$.  
It is important to note that in the case $\alpha>1$ we have that the parameter ${\gamma}$ equals the usual $\frac{1}{2}$, while for $\alpha\le 1$ it has the uncommon value $1-\frac{\alpha}{2}$. A similar dichotomy holds for the stationary counterpart $M^{(N)}.$ In the next subsection, we present explicit calculations for the mean and variance of $M^{(N)}(t)$ and $M^{(N)}$ that explain the reason behind this dichotomy. 

\subsection{Explicit calculations for the mean and variance}\label{expli}
We now present a number of explicit calculations for the mean and variance of the number of jobs present; for ease we consider the case that $\mu_i=\mu$ for all $i\in\{1,\ldots,d\}$, so that
Models {\sc i} and {\sc ii} coincide. We assume $J(0)$ is distributed according to 
the stationary distribution of the Markov chain $J(t)$. 
Directly from, e.g., \cite{BKMT}, for any $N\in{\mathbb N}$,
\[ \frac{{\mathbb E} M^{(N)}(t)}{N} = \varrho(t):= \frac{1-e^{-\mu t}}{\mu}\sum_{i=1}^d\pi_i\lambda_i,\:\:\:\:\:\:\:
 \frac{{\mathbb E} M^{(N)}}{N} = \varrho:= \frac{1}{\mu}\sum_{i=1}^d\pi_i\lambda_i.\]

We now concentrate on the corresponding variance; 
we first consider the non-scaled system, to later explore the effect of the time-scaling.
In the sequel we use
the notation $p_{ij}(t):={\mathbb P}(J(t)=j\,|\,J(0)=i).$
The `law of total variance', with $J\equiv (J( s))_{s=0}^t$, entails that
\begin{equation}\label{TC}\VAR\, M(t) = \MEAN\,\VAR (M(t)\,|\,J)+\VAR\,\MEAN(M(t)\,|\,J).\end{equation}
We first recall from (\ref{Pr})  that $M(t)$ obeys a Poisson distribution
with the {\it random} parameter $\varphi(J)$. 
As a result, the second term on the right of (\ref{TC}) can be written as 
\[\VAR \varphi(J) = \VAR\left(\int_0^t\lambda_{J(s)}e^{-\mu\,(t-s)}{\rm d}s\right) =
\int_0^t\int_0^t\COV\left( \lambda_{J(u)}e^{-\mu\,(t-u)} \lambda_{J(v)}e^{-\mu\,(t-v)}\right){\rm d}u{\rm d}v,\] 
which can be decomposed into
$I_1+I_2$, where
\begin{eqnarray*}
I_1&:=&\sum_{i=1}^d\sum_{j=1}^d\lambda_i\lambda_jK_{ij},\:\:\mbox{with}\:\:K_{ij}:=\int_0^t\int_0^v e^{-\mu(t-u)} e^{-\mu(t-v) }\pi_i\left(p_{ij}(v-u)-\pi_j\right){\rm d}u{\rm d}v,\\
I_2&:=&\sum_{i=1}^d\sum_{j=1}^d\lambda_i\lambda_jL_{ij},\:\:\mbox{with}\:\:L_{ij}:=\int_0^t\int_v^t e^{-\mu(t-u)} e^{-\mu(t-v) }\pi_j\left(p_{ji}(u-v)-\pi_i\right){\rm d}u{\rm d}v.
\end{eqnarray*}
Let us first evaluate $K_{ij}.$ To this end, substitute $w:=v-u$ (i.e., replace $u$ by $v-w$), and then interchange the order of integration, so as to obtain
\[K_{ij}= e^{-\mu t }\pi_i
\int_0^t\left(\int_w^t e^{2\mu v}{\rm d}v\right) e^{-\mu(t+w)} \left(p_{ij}(w)-\pi_j\right){\rm d}w.\]
Performing the inner integral (i.e., the one over $v$) leads to
\[K_{ij} =\frac{1}{2\mu} e^{-\mu t} \pi_i \int_0^t \left(e^{\mu(t-w)}-e^{-\mu(t-w)}\right)\left(p_{ij}(w)-\pi_j\right){\rm d}w.\]
The integral $L_{ij}$ can be evaluated similarly:
\begin{eqnarray*}
L_{ij}&=&e^{-\mu t}\pi_j \int_{0}^t\left(\int_0^{t-w}e^{2\mu v}{\rm d}v\right)
e^{-\mu(t-w)} \left(p_{ji}(w)-\pi_i\right){\rm d}w\\
&=&\frac{1}{2\mu}e^{-\mu t}\pi_j \int_{0}^t
\left(e^{\mu( t-w)}- e^{-\mu(t-w)}\right) \left(p_{ji}(w)-\pi_i\right){\rm d}w=K_{ji}.\end{eqnarray*}
The first term in the right hand side of (\ref{TC})
is easily evaluated, again relying on the fact that $M(t)$ has a Poisson distribution, conditional on $J$:
\[\MEAN\, \VAR (M(t)\,|\,J) =\sum_{i=1}^d \pi_i\lambda_i \int_0^t e^{-\mu s}{\rm d}s= \frac{1-e^{-\mu t}}{\mu}\sum_{i=1}^d\pi_i\lambda_i= \varrho(t).\]
Now we study the effect of the time-scaling: we replace  $\lambda_i$ by $N\lambda_i$ (for $i=1,\ldots,d$)
and $p_{ij}(w)$ by $p_{ij}(N^\alpha w)$ (for $i,j=1,\ldots,d$). Introduce the {\it deviation matrix} $D$, by
\[[D]_{ij} := \int_0^\infty \left(p_{ij}(t)-\pi_j\right){\rm d} t;\]
see e.g.\ \cite{CS}. 
Combining the above results, it is a matter of some elementary algebra to verify that, in obvious notation,
\[\VAR\,M^{(N)}(t)\sim N\varrho(t)
+N^{2-\alpha}\:\frac{1-e^{-2\mu t}}{\mu}
\sum_{i=1}^d\sum_{j=1}^d \pi_i{\lambda_i\lambda_j} [D]_{ij}.\]
From this relation, the above mentioned dichotomy becomes clear. It is observed that for $\alpha>1$ 
the variance of $M^{(N)}(t)$  grows linearly in $N$, and is essentially equal to
the corresponding mean, viz.\ $N{\varrho}(t).$ 
The intuition here is that in this regime the background process jumps faster than the arrival process, so that the arrival stream is nearly Poisson 
with parameter $\sum_{i=1}^d\pi_i\lambda_i$. The resulting system behaves therefore, as $N\to\infty$, essentially as an M/M/$\infty$.
If $\alpha<1$ the background process is slower than the arrival process. The variance of $M^{(N)}(t)$ now grows like $N^{2-\alpha}$,
proportionally to a constant that is a linear combination of the entries of the deviation matrix $D$. 

The above computations were done for the transient number of jobs $M^{(N)}(t)$, but obviously an analogous reasoning applies to its stationary counterpart $M^{(N)}$. 

\subsection{Preliminaries on deviation matrices, additional notation}\label{Dmat}
In this subsection, we recall a number of key properties of deviation matrices; for more detailed treatments
we refer to e.g.\ the standard texts \cite{KEIL,KEM,SYS}, as well as the compact survey  \cite{CS}. We also introduce additional notation,
which is intensively used later on.

We define the diagonal matrices $\Lambda$ and
$\M$, where $[\Lambda]_{ii} = \lambda_i$ and $[\M]_{ii}=\mu_i$. We
denote the invariant distribution corresponding to the transition
matrix $Q$ by the vector ${\m\pi}$; we follow the convention that vectors are column vectors unless stated otherwise, and that they
are written in bold fonts. As ${\m\pi}$ denotes the invariant distribution, we have ${\m\pi}^{\rm T}Q={\m 0^{\rm T}}$ and ${\m\pi}^{\rm T}{\m 1}={1}$, where ${\m 0}$ and ${\m 1}$ denote vectors of zeros and ones, respectively.
In the sequel we frequently use the `time-average arrival rate' 
$\lambda_\infty:=\sum_{i=1}^d\pi_i\lambda_i= {\m\pi}^{\rm T}{\Lambda{\m 1}}$, and
the `time average departure rate' 
$\mu_\infty:=\sum_{i=1}^d\pi_i\mu_i={\m\pi}^{\rm T}{\M{\m 1}}$.

We recall some concepts pertaining to the theory of deviation matrices of Markov processes; see e.g.\ \cite{CS}.
In particular, we let $\Pi := \m1 \bpi^{\rm T}$ denote the {\it ergodic matrix.}
We also define the {\it fundamental matrix} $F := (\Pi - Q)^{-1}$.
It turns out that the {deviation matrix} $D$, introduced above, satisfies $D= F - \Pi$.
We will frequently use the identities $QF=FQ=\Pi - I$, as well as the facts that $\Pi D= D \Pi =0$ (here $0$ is to be read as an all-zeros $d\times d$ matrix)
and $F{\m 1}={\m 1}.$ 

We use {the} following three vector-valued generating functions throughout the paper:
$\pgf$ denotes the unscaled probability generating function (pgf);
$\llgf\equiv\llgf\hN$ denotes the corresponding moment
generating function (mgf) under the law-of-large-numbers scaling; and
$\mgf\equiv\mgf\hN$ denotes the mgf {centered and normalized}
appropriately for the central limit theorem at hand.
For the transient cases, these generating functions
involve an extra argument $t$ to incorporate time. Importantly, all three generating functions
are {\it vectors} of dimension $d$ as we consider distributions jointly
with the state of  the background process; to make the notation easier, we assume that these vectors are {\it row} vectors. 
Lastly, $\phi\equiv\phi\hN$ denotes the scalar mgf under the 
{centering and normalization}
(obtained by summing the elements of $\mgf$).


\section{Outline of {\sc clt} proofs}
\label{sec:Proof}
In this section we point out how we set up our {\sc clt} proofs. In the next two sections this
`recipe' is then applied to analyze Model {\sc i} and Model {\sc ii}, covering both the transient and stationary number of jobs in the system.
We use a fairly classical approach to proving the {\sc clt}\,s for centered and normalized sequences of random variables
of the type 
(\ref{eq:GenCLT}). More specifically, our objective is to show that under the appropriate 
{normalization} (i.e., an appropriate choice of $\gamma$), the moment
generating function of (\ref{eq:GenCLT})
converges to that of the Normal distribution; the same is done for the stationary counterpart of (\ref{eq:GenCLT}).

Our technique consists of the following steps.
\begin{enumerate}[(a)]
  \item Derive a differential equation for the pgf $\pgf$ of the random quantities $M(t)$ and $M$.
  \item Establish the `mean behavior' $\vr(t)$ ($\vr$, respectively) of $M\hN(t)$ ($M\hN$, 
  respectively). This law of large numbers follows by manipulating 
  the  mgf $\llgf\equiv\llgf\hN$, obtaining a scalar limit solution $\exp(\vt\vr(t))$ in the transient case, and $\exp(\vt\vr)$
  in the stationary case.
  \item Reformulate the differential equation for the uncentered and unnormalized pgf $\pgf$ 
  into a recurrence relation for the
  centered and normalized mgf $\mgf\equiv\mgf\hN$.
  \item Manipulate and iterate this equation, approximate by suitable Taylor expansions, 
  to obtain a differential equation for the scalar mgf $\phi$ under the chosen  
  centering and normalization.
  \item  Discard asymptotically vanishing terms, so as to obtain a unique limit solution, viz.,
 $\phi(\vt) = \exp(\vt^2 \sigma^2(t))$ in the transient case
and $\phi(\vt) = \exp(\vt^2 \sigma^2)$   in the stationary case. We explicitly identify  
  $\sigma^2(t)$ and $\sigma^2$.
  \end{enumerate}
This limit solution resulting from the last step
corresponds to a zero-mean Normal distribution.
  Due to L\'evy's continuity theorem, this pointwise convergence of characteristic functions implies
  convergence in distribution to the zero-mean Normal random variable, so that we have derived the {\sc clt}.

 \noindent
{Issues related to the uniqueness of the solution of the differential equation are dealt with in
  Appendix A.}
Below we demonstrate this proof technique for the special case that the service rates in
each of the states are identical, i.e., $\M=\mu I$ for some $\mu>0$,
so that Models {\sc i} and {\sc ii} coincide. Importantly, Prop.\ \ref{prop:M-de-transient} in Section
\ref{subsec:Proof-diff-eq-p} holds for general $\M$.

\subsection{Differential equations for the pgf $\pgf$}
\label{subsec:Proof-diff-eq-p}

First we derive a system of differential equations for the pgf of the number of jobs in the system,
jointly with the background state.
We consider the bivariate process $(M(t), J(t))_{t\in {\mathbb R}}$, which is an {ergodic}
Markov process on the state space $\{1,\ldots,d\} \times \mathbb N$.
With the states of this process
enumerated in the obvious way, it has the (infinite-dimensional) transition rate matrix
\[
\begin{pmatrix}
    Q - \Lambda & \Lambda \\
\M & Q-\M-\Lambda & \Lambda \\
&2\M & Q-2\M-\Lambda & \Lambda \\
& &3\M & Q-3\M-\Lambda & \Lambda \\
& & & \ddots&\ddots&\ddots
\end{pmatrix}.
\]

We set out to find the transient distribution $(\pgf_k(t))_{k=0}^\infty$,
where $\pgf_k(t)$ is a $d$-dimensional row-vector whose entries are defined by
$[\pgf_k(t)]_j := {\mathbb P}(M(t) = k, J(t) = j).$
The (row-vector-)pgf $\pgf(t,z)$ is then defined through
\[ \pgf(t,z) := \sum_{k=0}^\infty \pgf_k(t) z^k, \]
such that 
\[[  \pgf(t,z)]_j ={\mathbb E} \left( z^{M(t)}1_{\{J(t)=j\}}\right).\]

\begin{prop}
\label{prop:M-de-transient}
The pgf $\pgf(t,z)$ satisfies the following differential equation:
\[ \pd{\pgf(t,z)}t = \pgf(t,z)\, Q + (z-1)\left(\pgf(t,z)\, \Lambda - \pd{\pgf(t,z)}z \M \right). \]
\end{prop}

\begin{proof} The result follows from classical arguments.
By virtue of the Chapman-Kolgomorov equation,
we have that
\begin{equation}\label{eq:de}
\frac{{\rm d}{\pgf_k(t)}}{{\rm d}t} = \pgf_{k-1}(t) \Lambda + \pgf_k(t) (Q - \Lambda - k \M) + (k+1) \pgf_{k+1}(t) \M,
\end{equation}
for all $k\in \mathbb N$, where we put $\pgf_{-1}(t):=0$ for all $t\geq0$.

From the standard relations
\[
 \sum_{k=0}^\infty (k+1) \pgf_{k+1}(t) z^k = \pd{\pgf(t,z)}z,
\:\:\mbox{and}\:\:
 \sum_{k=0}^\infty k \pgf_k(t) z^k = z \pd{\pgf(t,z)}z,\]
we obtain by multiplying both sides of (\ref{eq:de}) by $z^k$ and summing over $k\in\mathbb N$,
\[\pd{\pgf(t,z)}t =
 z \pgf(t,z) \Lambda + \pgf(t,z)(Q-\Lambda) - z \pd{\pgf(t,z)}z \M  + \pd{\pgf(t,z)}z \M. \]
The claim follows directly.
\end{proof}

We assume that at time
$0$ the system starts off empty.
Under the scaling $\Lambda \mapsto N \Lambda$ and $Q \mapsto N^{\alpha} Q$,
Prop.\ \ref{prop:M-de-transient} implies that
we have the following system of partial differential equations governing $(M\hN(t), J\hN(t))$:
\begin{equation}\label{MN-de-transient}
 \pd{\pgf\hN(t,z)}t = N^{\alpha} \pgf\hN(t,z)\, Q + (z-1)\left( N \pgf\hN(t,z) \Lambda - \pd{\pgf\hN(t,z)}z \M \right)
\end{equation}
describing the pgf $\pgf\hN$ of the number of jobs in the scaled system.

\subsection{Mean behavior}
\label{subsec: Proof-LLN}
In the remainder of this section, we assume that all service rates are identical:  $\M=\mu I.$
To obtain the limiting behavior of $N^{-1}\,M\hN(t)$ when $N$ grows large, it turns out to be convenient
to take the following steps.
\begin{enumerate}[(i)]
 \item Rewrite the differential equation (\ref{MN-de-transient}) as a recurrence relation for 
 $\pgf\hN$ involving the fundamental matrix $F$; recall from Section \ref{Dmat} the relation $QF=\Pi-I$.
 \item Translate this into a recurrence relation in terms of  the mgf $\llgf\hN$ of $N^{-1}\,M\hN(t)$, using
 a Taylor expansion for $z=\exp(\vt/N)$.
 \item Sum over the possible background states by postmultiplying with $\m1$, so as to obtain a scalar mgf;
 in this step we make use of the identity $F\m1=\m1$.
 \item Obtain the limiting differential equation by taking the limit for  $N \to \infty$. This equation
 has a closed solution. This is the mgf of the limiting constant $\vr(t)$.
\end{enumerate}

In this way we have proven the convergence in 
distribution of $N^{-1}M\hN(t)$ to $\varrho(t)$; as this limit is a constant, convergence in probability
follows immediately.

\vb

Let us go through the procedure in full detail now. 
Postmultiplication of Eqn.\ (\ref{MN-de-transient}) with $F$ and $N^{-\alpha}$, using
$QF=\Pi-I$, results in the recurrence relation
\begin{eqnarray}\label{pN-rec}
  \pgf\hN(t,z)& =& \pgf\hN(t,z)\,\Pi + N^{-\alpha} (z-1) \left(N \pgf\hN(t,z) \Lambda 
  - \pd{\pgf\hN(t,z)} z \M\right)F\\&& -\, N^{-\alpha} \pd{\pgf\hN(t,z)}t F.\nonumber
\label{eq:hr-tr-scaled-pgf}
\end{eqnarray}

We are now set to state and prove the mean behavior of $M\hN(t).$ Define
$\vr(t) := \vr\,(1-e^{-\mu t})$, with $\vr=\lambda_\infty/\mu$.
\begin{lem}\label{LEM1}
$N^{-1} M^{(N)}(t)$ converges in probability
to  $\vr(t)$, as $N\to\infty$.
\end{lem}

\begin{proof}
We introduce the transient scaled moment generating function $\llgf\hN(t,{\vartheta})$:
\[ \llgf\hN(t, {\vartheta}): = \pgf\hN(t, z), \]
with $z\equiv z\hN({\vartheta})=\exp({\vartheta}/ N)$. Evidently,
\[\pd{\llgf\hN(t,\vt)}t = \pd{\pgf\hN(t,z)}t,\qquad
\pd{\llgf\hN({t,\vt})}\vt =
 \pd{\pgf\hN(t,z)}z \od z {\vartheta} = \frac{z}{N} \,\pd{\pgf\hN(t,z)}z.\]
Substituting these expressions in Eqn.~(\ref{pN-rec}) and noting that 
$z^{\pm1}=1\pm\vt N^{-1}+O(N^{-2})$, we obtain
\begin{eqnarray*}
  \llgf\hN(t,\vt)& =& \llgf\hN(t,\vt) \Pi + N^{-\alpha} \left({\vartheta}\,\llgf\hN(t,\vt)\, \Lambda - {\vartheta} \,
  \pd{\llgf\hN(t,\vt)} {\vartheta} \mu I \right.\\
  &&\hspace{5cm}\left.-\, \pd{\llgf\hN(t,\vt)} t\right)F + o(N^{-\alpha}).
\end{eqnarray*}
The above implies that $\llgf\hN(t,\vt) = \llgf\hN(t,\vt)\Pi + O(N^{-\alpha})$, and
the same holds for the partial derivatives of $\llgf\hN(t,\vt)$, so all $\llgf\hN(t,\vt)$ 
between the brackets can be replaced by $\llgf\hN(t,\vt)\,\Pi$.
Postmultiplying by $\m1 N^\alpha$ and using the identities
$\Pi\m1=\m1$ and $F\m1=\m1$,  yields
\[
  0 = \left( \vt \lambda_\infty \llgf\hN(t, {\vartheta}) \,\m1 - \mu{\vartheta} \,\pd{ \llgf\hN(t, {\vartheta}) } {\vartheta}\m 1 -
   \pd{ \llgf\hN(t, {\vartheta}) } t\m 1 \right)  + o(1);
\] 
recall the definitions $\Pi:=\m1\m\pi^{\rm T}$ and
$\lambda_\infty:=\m\pi^{\rm T}\Lambda\m1$.
Define $\llgf(t,\vt)\m 1$ as the limit of $\llgf\hN(t,\vt)\m 1$
as $N\to\infty$.
Now multiply the differential equation with $N^\alpha$ and let $N\to\infty.$
We thus obtain a scalar partial differential equation in $\llgf(t,\vt)\m 1$
\[\pd{(\llgf(t,\vt)\m 1)}t = \vt\lambda_\infty(\llgf(t,\vt)\m 1) - \mu\vt\pd{(\llgf(t,\vt)\m 1)}\vt.\]
It is straightforward to check that 
  $\llgf(t,\vt)\m1 = \exp(\vt\varrho(t))$
satisfies the equation as well as the boundary conditions $\llgf(t,0)\m1=1$ and $\llgf(0,{\vartheta})\m 1=1$.
Now the stated follows directly.
\end{proof}

\subsection{Recurrence relations for the  centered and normalized mgf $\mgf\hN$}
\label{subsec: Proof-rec-rel-pt}

Now that we have derived the weak law of large numbers, we introduce in the
next step the centered and normalized mgf $\mgf\hN(t,\vt)$, that is, centered around
$N\vr(t)$ and normalized by $N^{-\gamma}$, with the scalar $\gamma$ yet to be determined.
We perform a change of variables in the recurrence relation for $\pgf\hN$,
Eqn.~(\ref{pN-rec}), so as to obtain the recurrence relation for the 
centered and normalized mgf $\mgf\hN$.

\vb

The pgf $\pgf\hN$ can be expressed in the normalized and centered mgf $\mgf\hN$
using
\[\mgf\hN(t,\vt) = \exp(-N\vr(t)\vt/N^\gamma)\, \pgf\hN\left(t, \exp(\vt/N^\gamma)\right), \]
which can be written as
\[ \pgf\hN(t,z) = \exp(\vr(t)\vt N^{1-\gamma}) \, \mgf\hN(t,\vt), \]
with $z\equiv z\hN(\vt) = \exp(\vt N^{-\gamma}).$
It is readily verified  that
\begin{eqnarray*} \pd{\pgf\hN(t,z)}z \od z \vt &=& \exp(\vr(t)\vt N^{1-\gamma}) 
\left( \vr(t) N^{1-\gamma} \mgf\hN(t,\vt) + \pd{\mgf\hN(\vt)}\vt \right);\\
 \od z \vt &=&  N^{-\gamma} \exp(\vt N^\gamma) =  N^{-\gamma} z, \end{eqnarray*}
so the derivatives of $\pgf\hN$ can be expressed in terms of the corresponding derivatives of $\mgf\hN$:
\begin{eqnarray*}\pd{\pgf\hN(t,z)}t &=&  \exp(\vr(t) \vt N^{1-\gamma})  \left( \vr'(t) \vt N^{1-\gamma}
\mgf\hN(t,\vt) + \pd{\mgf\hN(t,\vt)} t \right), \\
\pd{\pgf\hN(t,z)} z &=& \frac{1}{z}\, \exp(\vr(t)\vt N^{1-\gamma}) 
\left( N \vr(t)\, \mgf\hN(t,\vt) + N^\gamma \pd{\mgf\hN(t,\vt)}\vt \right). \end{eqnarray*}
Now perform the change of variables and substitute the expressions
for $\pgf\hN(t,z)$ and its partial derivatives  into Eqn.~(\ref{pN-rec}). Dividing by $\exp(\vr(t)\vt N^{1-\gamma})$ yields the following recurrence relation for $\mgf\hN$:
\begin{eqnarray}\label{ptN-rec}
\mgf\hN(t,\vt) &=& \mgf\hN(t,\vt) \Pi + 
N^{1-\alpha}\left(z\hN(\vt)-1\right) \mgf\hN(t,\vt) \Lambda F \nonumber\\
&&-N^{1-\alpha}\left(1-\frac{1}{z\hN(\vt)}\right) \vr(t) \,\mgf\hN(t,\vt)\M F \nonumber\\
&&-N^{\gamma-\alpha}\left(1-{\frac{1}{z\hN(\vt)}}\right) \pd{\mgf\hN(t,\vt)}\vt \M F\nonumber\\
&&-N^{1-\alpha-\gamma}\vr'(t)\vt\mgf\hN(t,\vt) F - N^{-\alpha}\pd{\mgf\hN(t,\vt)}t F.
\end{eqnarray}

\subsection{Differential equation for the scalar, centered and normalized, mgf $\phi\hN$}
\label{subsec: Proof-diff-eq-pt}

The next step is to expand $z$ in a Taylor series. Assuming certain restrictions
on $\gamma$ (that we later justify) we delete all terms of order smaller than $N^{-\alpha}$.
The resulting recurrence relation is iterated and manipulated until all
terms in the right-hand side contain $\mgf\hN\Pi$.
Next we postmultiply this system of partial differential equations by
$\m1$, so as to obtain a scalar partial differential equation in terms of
$\phi\hN(t,\vt) := \mgf\hN(t,\vt) \m1$. In this step we make use of the
definition of $\Pi:=\m1 \m\pi^{\rm T}$ and the identities $\Pi \m1=\m1$ and $F\m1=\m1$.

\vb 

The Taylor expansions of $z$ and $z^{-1}$ are
\[  z^{\pm 1} = 1 \pm \vt N^{-\gamma} + \frac12{\vt}^2 N^{-2\gamma} + O(N^{-3\gamma}),
\]
Applying these to Eqn.~(\ref{ptN-rec}) results in 
\begin{eqnarray}
\mgf\hN(t,\vt) &=& \mgf\hN(t,\vt) \Pi +
\vt N^{1-\alpha-\gamma}\mgf\hN(t,\vt)(\Lambda-\vr(t)\M-\vr'(t)I) F\nonumber\\
&&+\frac{\vt^2}2 N^{1-\alpha-2\gamma} \mgf\hN(t,\vt) (\Lambda + \vr(t)\M) F \nonumber\\
&&-\vt N^{-\alpha} \pd{\mgf\hN(t,\vt)}\vt \M F - N^{-\alpha}\pd{\mgf\hN(t,\vt)}t F+
O(N^{1-\alpha-3\gamma})+O(N^{-\alpha-\gamma}).\label{ptN-rec-taylored}
\end{eqnarray}
Under the assumption that $\gamma > 1/3$ (to be justified later) the order terms can be replaced by
$o(N^{-\alpha})$.

Next we iterate Eqn.\ (\ref{ptN-rec-taylored})
until all terms in the right-hand side either contain ${\mgf\hN(\vt)}\Pi$ or
are of $O(N^{-\alpha})$. For the latter we assume a second restriction, viz., $\gamma \ge 1-\alpha/2$
(also justified later).
We thus obtain
\begin{eqnarray*}
\mgf\hN(t,\vt) &=& \mgf\hN(t,\vt) \Pi \\
&&+\vt N^{1-\alpha-\gamma}\left(\mgf\hN(t,\vt) \Pi+ \vt N^{1-\alpha-\gamma}\mgf\hN(t,\vt)(\Lambda-\vr(t)\M-\vr'(t)I) F + \right.\\
&&\left.\qquad\qquad O(N^{1-\alpha-2\gamma}) + O(N^{-\alpha})\right)(\Lambda-\vr(t)\M-\vr'(t)I) F\\
&&+\frac{\vt^2}2 N^{1-\alpha-2\gamma}\left(\mgf\hN(t,\vt)\Pi + O(N^{1-\alpha-\gamma}) + O(N^{-\alpha}) \right)(\Lambda + \vr(t)\M) F \\
&&-\vt N^{-\alpha} \pd{\mgf\hN(t,\vt)}\vt \M F - N^{-\alpha}\pd{\mgf\hN(t,\vt)}t F+o(N^{-\alpha});
\label{eq:hr-tr-taylored-eq}
\end{eqnarray*}
here we remark that in the $O(N^{-\alpha})$-terms $\mgf\hN$ can be replaced by $\mgf\hN\Pi$ as an immediate consequence of the
fact that  
Eqn.\ (\ref{ptN-rec-taylored}) implies $\mgf\hN=\mgf\hN\Pi+o(1)$, while the same applies to
its derivatives. The above equation can be rewritten as
\begin{eqnarray}
\mgf\hN(t,\vt) &=& \mgf\hN(t,\vt) \Pi + \vt N^{1-\alpha-\gamma}\,\mgf\hN(t,\vt)\,\Pi(\Lambda - \vr(t) \M-\vr'(t)I) F \nonumber\\
&&+\vt^2 N^{2-2\alpha-2\gamma}\, \mgf\hN(t,\vt)\,\Pi(\Lambda - \vr(t) \M-\vr'(t)I) F (\Lambda - \vr(t) \M-\vr'(t)I) F \nonumber\\
&&+\frac{\vartheta^2}2 N^{1-\alpha-2\gamma} \mgf\hN(t,\vt)\,\Pi (\Lambda + \vr(t)\M) F \nonumber\\
\label{iter}&&-\vt N^{-\alpha} \pd{\mgf\hN(t,\vt)}\vt \,\Pi \M F - N^{-\alpha}\pd{\mgf\hN(t,\vt)}t  \,\Pi F+o(N^{-\alpha}).
  \end{eqnarray}
Now postmultiply Eqn.\ (\ref{iter}) by $ {\m 1}\, N^\alpha$; using the identities $\Pi{\m 1}={\m 1}$ and $F{\m 1}={\m 1}$, and the definition $\Pi:=\m1 \m\pi^{\rm T}$. We obtain
\begin{eqnarray*}
0 &=& \vt N^{1-\gamma} \phi\hN(t,\vt)\,{\m\pi}^{\rm T} \left(\Lambda - \vr(t) \M - \vr'(t)I\right) \m1\\
&&+\vt^2 N^{2-\alpha-2\gamma}\, \phi\hN(t,\vt)\,{\m\pi}^{\rm T}(\Lambda - \vr(t) \M-\vr'(t)I) F (\Lambda - \vr(t) \M-\vr'(t)I) \m 1\\
&&+\frac{\vt^2}2 N^{1-2\gamma} \phi\hN(t,\vt)\,{\m\pi}^{\rm T}(\Lambda + \vr(t) \M) \m1 - \vt \mu \pd{\phi\hN(t,\vt)}\vt - \pd{\phi\hN(t,\vt)}t + o(1).
\end{eqnarray*}
Directly from the definition of $\varrho(t)$, it is seen
that the first term on the right-hand side vanishes.
In addition, it takes some elementary algebra to check that
\[\m\pi^{\rm T}(\Lambda - (\vr(t)\mu+\vr'(t))I) F (\Lambda - (\vr(t)\mu+\vr'(t)I) \m1
=\m\pi^{\rm T}\Lambda D \Lambda \m1=:U,\]
where we used $F=D+\Pi=D+\m1\m\pi^{\rm T}$,
and
\[\frac12{\m\pi}^{\rm T}(\Lambda + \vr(t)\mu I) \m1= \lambda_\infty \left(1 - \frac{e^{-\mu t}}2 \right).\]
This results in the partial differential equation
\begin{eqnarray}
\lefteqn{\pd{\phi\hN(t,\vt)}t + \vt\mu\pd{\phi\hN(t,\vt)}\vt}\nonumber\\& =& 
\vt^2\phi\hN(t,\vt) \left(N^{2-\alpha-2\gamma}U +\frac12 N^{1-2\gamma} \lambda_\infty (1 - \frac12 e^{-\mu t}) \right) + o(1).\label{dvphi}
\end{eqnarray}

\subsection{Limit solution}
The last step in our proof is to obtain the limiting differential equation for $\phi(t,\vt)$, being the limit
of $\phi\hN(\vt,t).$ Its
unique solution corresponds to a normal distribution $\mathcal N(0,\sigma^2(t))$.

First, note that if we choose
$\gamma$ larger than both $1-\alpha/2$ and $1/2$, we do not obtain a {\sc clt},
but rather that the random variable under study converges in distribution to the constant $0$.
Hence, we take $\gamma=\max\{1-\alpha/2,1/2\}$, in which case the largest term
dominates in (\ref{dvphi}), with both terms contributing if $\alpha=1$. Note that this choice
is consistent with the restrictions on $\gamma$ we used during our proof.
We obtain by sending $N\to\infty$, 
\begin{equation}
\pd{\phi(t,\vt)}t + \vt \mu \pd{\phi(t,\vt)}\vt = \vt^2\phi(t,\vt)\,g(t),
\end{equation}
with $g(t) := U\, 1_{\{\alpha\le 1\}} + (\lambda_\infty (1 - e^{-\mu t}/2))\, 1_{\{\alpha\ge 1\}}$.

We propose the {\it ansatz} \[ \phi(t,\vt) = \exp\left(\frac12 \vt^2 e^{-2\mu t} f(t)\right), \] for some unknown function $f(t)$; recognize the mgf associated with the Normal distribution.
This leads to the following ordinary differential equation for $f(t)$:
\[ f'(t) = 2 e^{2\mu t} g(t), \]
which is obviously solved by integrating the right-hand side.
From this we immediately find the expression for the variance $\sigma^2(t)$ of the Normal distribution.

\vb

With this last step we have proven our claim. It is instructive to compare the findings with the expressions obtained in Section \ref{expli}.
\begin{thm} \label{TH1}
Consider Model {\sc i} or {\sc ii} with $\mu_i=\mu$ for all $i\in\{1,\ldots,d\}$.
The random variable
\[\frac{M^{(N)}(t) - N\vr(t)}{N^\gamma}\]
converges to a Normal distribution with zero mean and variance $\sigma^2(t)$ as $N\to\infty$;
here the parameter $\gamma$ equals $\max\{1-\alpha/2,1/2\}$, and $\sigma^2(t):=\sigma_m^2(t)1_{\{\alpha\le 1\}}+\varrho(t)1_{\{\alpha\ge 1\}},$
with
$\sigma_m^2(t) := \mu^{-1}(1- e^{-2\mu t})U$.
\end{thm}

\begin{cor} \label{TH1s}
Consider Model {\sc i} or {\sc ii} with $\mu_i=\mu$ for all $i\in\{1,\ldots,d\}$.
The random variable
\[\frac{M^{(N)} - N\vr}{N^\gamma}\]
converges to a Normal distribution with zero mean and variance $\sigma^2$ as $N\to\infty$;
here the parameter $\gamma$ equals $\max\{1-\alpha/2,1/2\}$, and 
$\sigma^2:=\sigma_m^2 1_{\{\alpha\le 1\}}+\varrho1_{\{\alpha\ge 1\}},$
with $\sigma_m^2 := \mu^{-1} U$.
\end{cor}

\section{Model {\sc i}: Stationary and transient distribution}
\label{sec:hazard-diff-eq}
In this section we briefly recall the results of the steps for Model {\sc i},
both for the stationary and time-dependent behavior.
The proofs are analogous to those in the previous section. Comparing the results of Lemma \ref{LEM3} and
Thm.\ \ref{THM1t} with those of Lemma 
\ref{LEM1} and Thm.\ \ref{TH1}, respectively, the effect of heterogeneous service rates becomes visible.

\begin{prop}
\label{prop:de-hazard}
Consider Model \,{\sc i}. In the stationary case
the pgf $\pgf(z)$ satisfies the following differential equation:
\[
    \pgf(z) Q = (z-1) \left(\od{\pgf(z)}z \M - \pgf(z) \Lambda\right).
\]
In the transient case the pgf $\pgf(t,z)$ satisfies the following differential equation:
\[ \pd{\pgf(t,z)}t = \pgf(t,z)\, Q + (z-1)\left(\pgf(t,z)\, \Lambda - \pd{\pgf(t,z)}z \M \right). \]
\end{prop}

Define ${\varrho}\hI: = \lambda_\infty/\mu_\infty$, and ${\varrho}\hI(t) = \varrho\hI\,(1-e^{-\mu_\infty t})$.

\begin{lem}\label{LEM2} Consider Model\, {\sc i}.
As $N\rightarrow \infty$,
\begin{itemize}
\item[(1)] $N^{-1} M^{(N)}(t)$ converges in probability to  ${\varrho}\hI(t)$.
\item[(2)] $N^{-1} M^{(N)}$ converges in probability to ${\varrho}\hI$.
\end{itemize}

\end{lem}

\begin{thm}\label{THM1s} Consider Model\, {\sc i}.
The random variable
\[{\frac{M^{(N)}(t) - {N}{\varrho\hI(t)}}{N^{\gamma}}}\]
converges to a Normal distribution with zero mean and variance $\sigma^2(t)$ as $N\to\infty$;
here\\
$\sigma^2(t):=\sigma_m^2(t)1_{\{\alpha\le 1\}}+\varrho\hI(t)1_{\{\alpha\ge 1\}},$
with
\[ \sigma_m^2(t) := 2 e^{-2\mu_\infty t} \int_0^t e^{2\mu_\infty s}\bpi^{\rm T} (\Lambda - {\varrho}\hI(s) \M) D (\Lambda - {\varrho}\hI(s) \M) \m1\,{\rm d}s. \]
The random variable
\[ {\frac{M^{(N)} - {N}{\varrho}\hI}{N^{\gamma}}}\]
converges to a Normal distribution with zero mean and variance $\sigma^2$ as $N\to\infty$;
here\\
$\sigma^2:=\sigma_m^2 1_{\{\alpha\le 1\}}+{\varrho}\hI1_{\{\alpha\ge 1\}}$, with \[\sigma_m^2 :=  \mu_\infty^{-1} \bpi^{\rm T}(\Lambda -{\varrho}\hI \M) D(\Lambda-{\varrho} \hI\M) \m 1.\]
In both cases the parameter $\gamma$ equals $\max\{1-\alpha/2,1/2\}$.
\end{thm}

The formula for $\sigma^2_m(t)$ can be evaluated more explicitly. Define $G_{m,n}(t):=e^{-m\mu_\infty t}-e^{-n\mu_\infty t}$, for $m,n\in{\mathbb N}$. 
Direct computations yield that $\sigma_m^2(t)$ equals
\[U\frac{1}{\mu_\infty}G_{0,2}(t)+ \hat U\,\frac{\varrho\hI}{\mu_\infty}\left(2G_{1,2}(t)-G_{0,2}(t)\right) +\check U\, \frac{(\vr\hI)^2}{\mu_\infty}\left(G_{0,2}(t)-4 G_{1,2}(t)+2\mu_\infty t e^{-2\mu_\infty t }\right),\]
with $\hat U:= \m\pi^{\rm T}{\M}D\Lambda\m 1+ \m\pi^{\rm T}\Lambda D{\M}\m 1$ and $\check U:=\m\pi^{\rm T}{\M}D{\M}\m 1.$
It is readily verified that $\sigma_m^2(t)\to\sigma^2_m$ as $t\to
\infty$, as expected. 

\section{Results for Model {\sc ii}}
\label{sec:Mod2}
In this section we study Model {\sc ii}: the service
times are now determined by the background state as seen by the job upon  arrival.
The approach is as before: we first derive a system of differential equations
(Section \ref{pgfII}), 
then establish the mean behavior by means of laws of large numbers (Section \ref{llnII}), and
finally derive the {\sc clt}\,s (Section \ref{cltII}).

\subsection{Differential equations for the pgf $\pgf$.}\label{pgfII}
For the transient distribution, a system of differential equations
was previously derived in \cite{BKMT}.
{It} is based on the observation that $M(t)$ has a Poisson distribution
with (random) parameter $\varphi(J)$, see (\ref{Pr}). 
The intuition behind this formula is that a job arriving at time $s$ survives
in the system until time $t$ with probability $e^{-\mu_i\,(t-s)}$ (assuming
that the background process is in state $i$), which is distributionally
equivalent with `thinning' the Poisson parameter with exactly this fraction.
This description yields,  after some manipulations, the following
differential equation for the pgf, the row vector $\pgf(t,z)$:
\begin{equation}
\label{dvII} \frac{\partial {\boldsymbol p}(t,z)}{\partial t} = \pgf(t,z) \tilde Q+ (z-1) \pgf(t,z) \Delta(t), \end{equation}
where $\tilde Q =(\tilde q_{ij})_{i,j=1}^d$ is the transition rate matrix of the time-reversed version of $J(\cdot)$ (i.e., $\tilde q_{ij}:= q_{ji}\pi_j/\pi_i$),
and $\Delta(t)$ denotes
a diagonal matrix with entries $[\Delta(t)]_{ii} := \lambda_i \exp(-\mu_it)$.

\begin{rem} {\em It is noted that the definition of $\boldsymbol p$ is slightly different from the one used in \cite{BKMT}. In the present paper 
we consider the generating function of the number of jobs present at time $t$
\emph{jointly with the state of the background process at time $t$}, whereas \cite[Prop.\ 2]{BKMT} considers 
the generating function of the number of jobs present at time $t$
 \emph{conditioned on the background state at time $0$}.
As a consequence, we obtain a slightly different equation, but it is easy to translate them into each other.
\hfill$\diamondsuit$}\end{rem}

Our objective is to set up our proof such that it facilitates proving both the transient and stationary {\sc clt}. 
Na\"ively, one could try to obtain a differential equation for the stationary behavior by sending $t\to\infty$ in (\ref{dvII}), but it is readily checked that this yields a trivial relation only: ${\boldsymbol 0} ={\boldsymbol 0} .$ A second na\"ive approach would be to establish the {\sc clt} for $M^{(N)}(t)$, and to send then $t$ to $\infty$; it is clear, however, that this procedure relies on interchanging two limits
($N\to\infty$ and $t\to\infty$), of which a formal justification is lacking.

We therefore resort  to an alternative approach. It relies on 
a description based on a more general state space: we do not only keep track of the 
number of jobs present,
but we rather record  the
numbers of jobs present {\it of each type}, where `type' refers to the state of the background process upon arrival. To this end, we introduce the $d$-dimensional stochastic
process \[{\boldsymbol M}(t)=\left(M_1(t), \ldots, M_d(t)\right)_{t\in\mathbb R},\]
where the $k$-th entry denotes the number of particles of type $k$ in the
system at time~$t$. The transient and stationary  total numbers of jobs present are denoted by
\[M(t):= \sum_{k=1}^d M_k(t),\:\:\:\:\:M:= \sum_{k=1}^d M_k,\]
respectively. As usual,  we add a superscript $\hN$ when working with the model in which 
imposed our scaling on the arrival rates and the transition rates of the background process.

\newcommand{\vz}{{\m z}}

\vb

As before, we first derive a differential equation for the unscaled model.
The generating function $\pgf(t,\vz)$ is defined as follows:
\[ [\pgf(t,\vz)]_j = \ee\left( \prod_{k=1}^d z_k^{M_k(t)} 1_{\{J(t)=j\}} \right). \]
In addition, $E_k$ is a matrix for which $[E_k]_{kk} = 1$, and whose other entries are zero.
For a row vector ${\m q}$, the
multiplication ${\m q}\, E_k$ thus results in a (row) vector which leaves the $k$-th entry of ${\m q}$ unchanged
while the other entries become zero.
The following result covers the transient case.

\begin{prop} \label{pgfIIt} Consider Model \,{\sc ii}.
The pgf $\pgf(t,\vz)$
satisfies the following differential equation:
\[ \pd{\pgf(t,\vz)}t = \pgf(t, \vz) Q + \sum_{k=1}^d (z_k-1) \left(\lambda_k \,\pgf(t,\vz)\, E_k- \mu_k \pd{\pgf(t,\vz)}{z_k}\right). \]
\end{prop}

With the pgf $\pgf(z_1,\ldots,z_d)$ defined in the obvious way,
the differential equation for the stationary case is the following.

\begin{prop} \label{pgfIIs} Consider Model \,{\sc ii}.
The pgf $\pgf(\vz)$ satisfies the following differential equation:
\[ 0 = \pgf(\vz) Q + \sum_{k=1}^d (z_k-1) \left( \lambda_k\, \pgf(\vz)\, E_k-\mu_k \pd{\pgf(\vz)}{z_k}\right). \]
\end{prop}
The proofs of these propositions are straightforward, and follow the same lines as before: we consider
the generator of the Markov process, and transform the Kolmogorov equation (for
the transient case) and the invariance equation (for the stationary case).

The partial differential equation for the transient scaled model follows directly from Prop.~\ref{pgfIIt}, by replacing
$\lambda_k$ by $N\lambda_k$, and $Q$ by $N^\alpha Q$. It results in
\begin{equation}\label{MN2-de-transient}
 \pd{\pgf\hN(t,\vz)}t = N^{\alpha} \pgf\hN(t,\vz)\, Q + 
 \sum_{k=1}^d (z_k-1) \left(N\lambda_k \,\pgf\hN(t,\vz)\, E_k-
 \mu_k \pd{\pgf\hN(t,\vz)}{z_k}\right).
\end{equation}
The stationary case can be dealt with analogously, relying on Prop.\ \ref{pgfIIs}.

\vb

Our objective is to
derive the {\sc clt} for both the transient and stationary case. We do so by presenting the full analysis for the transient case; in the stationary case we can leave out one term. 
Importantly, this approach does not have the problem of illegitimately interchanging two limits.

\subsection{Mean behavior}\label{llnII}
As before, we first derive the law of large numbers.
Again we rewrite the differential equations (\ref{MN2-de-transient}) as a recurrence relation
for $\pgf\hN$ that involves the fundamental matrix $F$:
\begin{eqnarray} \pgf\hN(t, \vz) &=& \pgf\hN(t, \vz) \Pi +N^{-\alpha} \sum_{k=1}^d (z_k-1) \left(N \lambda_k\, \pgf\hN(t,\vz)\, E_k  - \mu_k \pd{\pgf\hN(t,\vz)}{z_k}\right) F\nonumber\\&& -\, N^{-\alpha} \pd{\pgf\hN(t,\vz)}t F \label{fund}\end{eqnarray}
for the transient case,
and likewise for the stationary case.

The following lemma  establishes weak laws of large numbers for ${\boldsymbol M}^{(N)}(t)$ and $M\hN(t)$, as well as their steady-state counterparts
$ {\boldsymbol M}^{(N)}$ and $M\hN$. 
We first define 
\[{\varrho_k\hII(t)} := \pi_k \frac{\lambda_k}{\mu_k} \, (1 - e^{-\mu_k t}),\:\:\:\:\:
{\varrho\hII_k} := \pi_k \frac{\lambda_k}{\mu_k}.\] 
Also, 
$\varrho\hII(t):=\sum_k \varrho\hII_k(t)$ 
and $\varrho\hII:=\sum_k \varrho\hII_k.$

\begin{lem}\label{LEM4} Consider Model \,{\sc ii}. As $N\to\infty$,
\begin{itemize}
\item[(1)]
$N^{-1} {\boldsymbol M}^{(N)}(t)$ converges in probability
to ${\boldsymbol\varrho}\hII(t)$.
\item[(2)]
$N^{-1} {\boldsymbol M}^{(N)}$ converges in probability to
${\boldsymbol\varrho}\hII$.
\item[(3)] $N^{-1}M\hN(t)$ converges in probability to $\varrho\hII(t)$,
and $N^{-1}M\hN$ to $\varrho\hII$.
\end{itemize}
\end{lem}

\newcommand{\vth}{{\m\vt}}

\begin{proof}
Similarly to the proof of Lemma \ref{LEM1}, we first introduce the scaled moment generating function $\llgf\hN(t,\vth) :=
\pgf\hN(t,\vz),$
with $z_k\equiv z_k\hN({\vartheta}_k)=\exp({\vartheta}_k /N)$, for $k=1,\ldots,d$. We see immediately that
\[ \pd{\llgf\hN(t,\vth)} t = \pd{\pgf\hN(t,\vz)} t,\:\:\:\:\:\:
\pd{\llgf\hN(t,\vth)}{{\vartheta}_k} =  \pd{\pgf\hN(t,\vz)} {z_k} \od{z_k}{{\vartheta}_k} = \frac{z_k}{N} \,\pd{\pgf(t,\vz)}{z_k}. \]

Now we substitute these expressions in Eqn.\ \eqref{fund}, and note
that $z_k^{\pm1} = 1\pm {\vartheta}_k N^{-1} + O(N^{-2})$. As a consequence,
\begin{eqnarray*} \llgf\hN(t,\vth) &=& \llgf\hN(t,\vth) \Pi+N^{-\alpha} \sum_{k=1}^d {\vartheta}_k \left(\lambda_k \,\llgf\hN(t,\vth)\, E_k - \mu_k \pd{{\llgf}\hN(t,\vth)}{{\vartheta}_k}\right) F\\&& -\,  N^{-\alpha} \pd{{\llgf}\hN(t,\vth)}t F + o(N^{-\alpha}). \end{eqnarray*}
It directly follows  that $ \llgf\hN(t,\vth) = \llgf\hN(t,\vth)\Pi + O(N^{-\alpha})$, and hence also 
\[\frac{\partial\llgf\hN(t,\vth)}{\partial t} = \frac{\partial\llgf \hN(t,\vth)}{\partial t}  \,\Pi + O(N^{-\alpha}),\:\:\:\:\:
\frac{\partial \llgf\hN(t,\vth)}{\partial {\vartheta}_k} =\frac{\partial\llgf \hN(t,\vth)}{\partial {\vartheta}_k}\,\Pi + O(N^{-\alpha}).\] 
The next step is to postmultiply the previous display by $\m 1\, N^{\alpha}$, 
and after some elementary  steps we obtain the following scalar partial differential equation in $\llgf\hN(t,\vth) \m1$:
\[ \pd{(\llgf\hN(t,\vth) \m1)}t = \sum_{k=1}^d \vt_k \left(
\pi_k \lambda_k(\llgf\hN(t,\vth) \m1)-\mu_k \pd{(\llgf\hN(t,\vth) \m1)}{\vt_k}
\right) + o(1). \]
Now let $N\to\infty$; define $\llgf(t,\vth) \m1:=\lim_{N\to\infty} \llgf\hN(t,\vth) \m1 .$
We propose the following form for the limiting function $\llgf(t,\vth) \m1$:
\[ \llgf(t,\vth) \m1 = \exp\left( \sum_{k=1}^d {\vartheta}_k \bar\varrho_k(t) \right), \]
for specific functions $\bar\varrho_k(\cdot)$ (to be determined later).
Plugging this form into the differential equation, it means that the following
equation must be fulfilled by the $\bar\varrho_k(\cdot)$:
\[ \sum_{k=1}^d {\vartheta}_k \left(  \bar\varrho'_k(t) -  \pi_k \lambda_k + \mu_k \bar\varrho_k(t) \right) = 0. \]
As this must hold for any ${\vartheta}_k$, this equation leads to a separate differential equation for
every $\bar \varrho_k(t)$, which moreover agrees with the one in the first part of the claim  ($\bar \varrho_k(t)=\varrho\hII_k(t)$, that is). We conclude that we have established the claim for the transient case: $N^{-1} {\boldsymbol M}^{(N)}(t)$ converges in probability
to ${\boldsymbol\varrho}\hII(t)$  as $N\to\infty.$

For the stationary case, we can follow precisely the same procedure, but without
the partial derivative with respect to time, so that we now end up with a differential equation
in $\llgf(\vth) \m1$ as follows:
\[ 0 = \sum_{k=1}^d \vt_k \left(\pi_k \lambda_k (\llgf(\vth) \m1) - 
\mu_k \pd{(\llgf(\vth) \m1)}{\vt_k}\right) , \]
for which $\llgf(\vth) \m1 = \exp(\sum_{k=1}^d {\vartheta}_k \varrho_k \hII)$ forms
a solution. This completes the proof of the second claim. The third claim follows trivially.
 \end{proof}

\subsection{Central limit theorems}\label{cltII}

Next, we state and prove the {\sc clt} result for Model {\sc ii}.
To this end, we first define the (symmetric) matrices $V(t)$ and $V:=\lim_{t\to\infty} V(t)$ with entries
\[ [V(t)]_{jk} := \frac{\lambda_j\lambda_k[\bar D]_{jk}}{\mu_j + \mu_k} (1 - e^{-(\mu_{{j}}+\mu_{{k}})t}),\:\:\:\:\:
[V]_{jk}=\frac{\lambda_j\lambda_k[\bar D]_{jk}}{\mu_j + \mu_k} ;\]
here $\bar D$ denotes the (symmetric) matrix defined by~$[\bar D]_{jk} = (\pi_j [D]_{jk} + \pi_k [D]_{kj})$. Also, $C:=\lim_{t\to\infty}C(t)$, where
\[[C(t)]_{jk}:=[V(t)]_{jk}1_{\{\alpha\le 1\}}+\varrho\hII_j(t)1_{\{\alpha\ge 1\}}1_{\{j=k\}}.\]

The following theorem is the main result of this section. 

\begin{thm}\label{THMM2} Consider Model \,{\sc ii}.
The random vector
\[ {\frac{{\boldsymbol M}\hN(t) - {N}{\boldsymbol \varrho}\hII(t)}{N^\gamma}}\]
converges to a $d$-dimensional Normal distribution with zero mean and covariance matrix $C(t)$ as $N\to\infty$.
In both cases the parameter $\gamma$ equals $\max\{1-\alpha/2,1/2\}$.
The random vector
\[{\frac{{\boldsymbol M}\hN - {N}{\boldsymbol \vr}\hII}{N^\gamma}}\]
converges to a $d$-dimensional Normal distribution with zero mean and covariance matrix $C$ as $N\to\infty$.
\end{thm}

\begin{proof}
Mimicking the proof of the {\sc clt} in Section \ref{sec:Proof}, we start again with setting up
a recurrence relation for the centered and normalized mgf $\mgf\hN$.
Define $\vz$ by $z_k\equiv z\hN_k(\vt_k) :=\exp(\vt_k N^{-\gamma})$, for
$k=1,\ldots, d$, with the value of $\gamma$ to be determined later on.
We first concentrate on the transient case and introduce the
centered and normalized mgf $\mgf(t,\vth)$:
\[\mgf\hN(t,\vth) = \exp\left(-N^{1-\gamma} \sum_{k=1}^d {\vartheta}_k \varrho\hII_k(t)\right) \pgf\hN\left(t,\vz \right).\]
We wish to perform a change of variables in Eqn.~(\ref{fund}) to obtain
a recurrence relation in $\mgf\hN(t,\vth)$.
To this end, note that
\[\pd{\pgf\hN(t,\vz)} {z_k} \od{z_k}{{\vartheta}_k} = 
\exp\left(N^{1-\gamma} \sum_{k=1}^d \vt_k \vr\hII_k(t)\right) 
\left( \vr\hII_k(t) N^{1-\gamma} \mgf\hN(t,\vth) + \pd{\mgf\hN(t,\vth)}{\vt_k} \right),
\]
where
\[ \od {z_k}{{\vartheta}_k} =  N^{-\gamma} \exp(\vt_k N^{-\gamma}) =  N^{-\gamma} z_k. \]
Also,
\[ \pd{\pgf\hN(t,\vz)}  t = \exp\left(N^{1-\gamma} \sum_{k=1}^d \vt_k \vr\hII_k(t)\right)
\left( \sum_k {\vt_k}\afg N^{1-\gamma} \mgf\hN(t,\vth) + \pd{\mgf\hN(t,\vth)}t \right). \]
Now perform the change of variables, and substitute the expressions for the partial 
derivatives of $\pgf\hN(t,\vz)$ into Eqn.~(\ref{fund}). 
Dividing the equation by $\exp(N^{1-\gamma} \sum_{k=1}^d \vt_k \vr\hII_k(t))$ gives the following recurrence
relation for $\mgf\hN(t,\vz)$:
\begin{eqnarray*}\mgf\hN(t,\vth) &=&\mgf\hN(t,\vth)\Pi + 
N^{1-\alpha} \sum_{k=1}^d (z_k-1) \lambda_k \,\mgf\hN(t,\vth)\, E_k \,F\\
&&-\, N^{-\alpha} \sum_{k=1}^d \left(1-\frac{1}{z_k}\right) N^\gamma \mu_k 
\left(N^{1-\gamma}\vr\hII_k(t)\mgf\hN(t,\vth)+\pd{\mgf\hN(t,\vth)}{{\vartheta}_k}\right) F \\
&&- \,N^{1-\alpha-\gamma} \sum_{k=1}^d {\vartheta}_k \afg\mgf\hN(t,\vth) F - N^{-\alpha} \pd{\mgf\hN(t,\vth)}t F. \end{eqnarray*}
The next step is to introduce the second order Taylor expansions for $z_k$ and $z_k^{-1}$:
\[z_k^{\pm 1} = 1 \pm \vt_k N^{-\gamma} + \frac12\vt_k^2 N^{-2\gamma} + O(N^{-3\gamma}).\]
Ignoring all terms that are provably smaller than $N^{-\alpha}$ under the assumption
that $\gamma > 1/3$ (justified later), and combining terms of the same order, we obtain
\begin{eqnarray*}
\mgf\hN(t,\vth)  \hspace{-1mm}&=&\hspace{-1mm}\mgf\hN(t,\vth) \,\Pi + N^{1-\alpha-\gamma} \sum_{k=1}^d {\vartheta}_k \,\mgf\hN(t,\vth)  \left(\lambda_k E_k - \mu_k \varrho\hII_k(t) I - \afg I\right)\hspace{-0.5mm} F \\
&&+\, N^{1-\alpha-2\gamma} \sum_{k=1}^d\frac{ {\vartheta}_k^2}{2} \mgf\hN(t,\vth) \left(\lambda_k E_k + \mu_k\varrho\hII_k(t) I \right){F}
\\&&-\, N^{-\alpha} \sum_{k=1}^d {\vartheta}_k \mu_k \pd{\mgf\hN(t,\vth)}{{\vartheta}_k} F - N^{-\alpha} \pd{\mgf\hN(t,\vth)} t F,
\end{eqnarray*}
up to an error term that is $o(N^{-\alpha}).$
As we did in the proof of the {\sc clt} in Section \ref{sec:Proof} with 
Eqn.~(\ref{ptN-rec-taylored}), we iterate and manipulate this relation, under the 
assumption that $\gamma \geq 1-\alpha/2$ (justified later), until all terms in the right-hand side contain
$\mgf\hN\Pi$. Then we postmultiply with $\m 1 \, N^\alpha$, and develop a differential equation in terms of $\phi\hN(t,\vth):=\mgf\hN(t,\vth)\,{\m 1}.$
After some (by now quite familiar) manipulations, we obtain the following partial differential equation in $\phi\hN(t,\vth)$:
\begin{eqnarray*}
\lefteqn{\hspace{-10mm}\pd{\phi\hN(t,\vth)} t + 
\sum_{k=1}^d \vt_k \mu_k \pd{\phi\hN(t,\vth)}{\vt_k} = \frac12 \phi\hN(t,\vth) \left( 
N^{2-\alpha-2\gamma} \sum_{j=1}^d\sum_{k=1}^d \vt_j\vt_k {\lambda_j\lambda_k} [\bar D]_{jk}\right.}\\
&&
\left.\hspace{5cm} +\,N^{1-2\gamma} \sum_{k=1}^d \vt_k^2 \pi_k (\lambda_k + \mu_k\varrho\hII_k(t))
\right) +o(1),\end{eqnarray*}
where we have used that
\begin{eqnarray*}
 \lefteqn{\hspace{-1.9cm} \bpi^{\rm T} \left( \sum_{j=1}^d \vt_j (\lambda_j E_j - 
 \mu_j \vr\hII_j(t) I - \afg I)\hspace{-0.7mm}  \right)\hspace{-0.7mm} F 
 \hspace{-0.7mm} \left( \sum_{k=1}^d \vt_k (\lambda_k E_k - \mu_k \vr\hII_k(t) I - \afg I) \right)\hspace{-0.7mm}  \m 1}\\
       &=&\sum_{j=1}^d\sum_{k=1}^d {\vartheta}_j {\vartheta}_k \lambda_j \lambda_k\left( \bpi ^{\rm T}E_j D E_k \m 1\right) = \frac12 \sum_{j=1}^d\sum_{k=1}^d {\vartheta}_j {\vartheta}_k \lambda_j \lambda_k [\bar D]_{jk}.
\end{eqnarray*}

The last part of the proof concerns the limiting behavior as $N\to\infty.$
Pick, as before, $\gamma=\max\{1-\alpha/2,1/2\}$, to obtain the  following partial differential equation:
\begin{eqnarray*}
\lefteqn{\pd{\phi(t,\vth)} t + \sum_{k=1}^d {\vartheta}_k \mu_k \pd{\phi(t,\vth)}{{\vartheta}_k} }\\
&=&\frac12 \phi(t,\vth)  \left( \sum_{j=1}^d \sum_{k=1}^d {\vartheta}_j{\vartheta}_k {\lambda_j\lambda_k}[\bar D]_{jk} 1_{\{\alpha\leq1\}}+ \sum_{k=1}^d {\vartheta}_k^2 (\pi_k \lambda_k + \mu_k\varrho\hII_k(t)) 1_{\{\alpha\geq1\}} \hspace{-0.7mm}\right)\hspace{-1mm}. \end{eqnarray*}
It is straightforward to verify that the following expression
constitutes a solution for this differential equation:
\[\phi(t,\vth) = \exp\left(
\frac12 \sum_{j=1}^d\sum_{k=1}^d \vt_j\vt_k [V(t)]_{jk}1_{\{\alpha\leq1\}}
+ \frac12 \sum_{k=1}^d \vt_k^2 \vr\hII_k(t)1_{\{\alpha\geq1\}}
\right).\]

If we redo the derivation for the stationary case (i.e., we now discard the terms
originating from the derivative with respect to $t$ in the original partial differential equation), we end up with
\[ \phi(\vth) = \exp\left( \frac12\sum_{j=1}^d \sum_{k=1}^d{\vartheta}_j{\vartheta}_k [V]_{jk}1_{\{\alpha\leq1\}}+\frac12 \sum_{k=1}^d {\vartheta}_k^2 \varrho\hII_k1_{\{\alpha\geq1\}}  \right).\]
This completes the proof. \end{proof}

\begin{cor} Consider Model\, {\sc ii}.
An immediate consequence of Thm.\ $\ref{THMM2}$ is that, with $\gamma$ as defined before, the random variables
\[{\frac{M\hN - N\vr\hII}{N^\gamma}} \:\:\:\mbox{and}\:\:\:\:
{\frac{M\hN(t) - N\vr\hII(t)}{N^\gamma}}\]
converge to Normal distributions with zero mean and variances
\[\sum_{j=1}^d \sum_{k=1}^d [V]_{jk}1_{\{\alpha\leq1\}}+ \varrho\hII 1_{\{\alpha\geq1\}}\:\:\:\mbox{and}\:\:\:\:\sum_{j=1}^d \sum_{k=1}^d [V(t)]_{jk}1_{\{\alpha\leq1\}}+ \varrho\hII(t) 1_{\{\alpha\geq1\}},\]
respectively, as $N\to\infty$.\end{cor}

\section{Correlation across time}\label{CAT}
Above we analyzed the joint distribution of the two queues at a given point in time. A related question, to be covered in this section, concerns the joint distribution at distinct time epochs. For ease we assume that the service rates are identical (and equal to $\mu$), so that Model {\sc i} and Model {\sc ii} coincide.

\subsection{Differential equation}
We follow the line of reasoning of \cite[Prop.\ 2]{BKMT}; we consider again the non-scaled model, but, as before, these results can be trivially translated in terms of the $N$-scaled model.
Fix time epochs $0\equiv s_1\le s_2\le\cdots\le s_K$ for some $K\in{\mathbb N}$. The goal of this subsection
is to characterize the joint transform, for $j=1,\ldots,d$,
\[\Psi_j({t},\vz):={\mathbb E}\left(\left.\prod_{k=1}^K z_k^{ M(t+s_k)}\,\right| \,J(0)=j\right).\]Assume a job arrives between $0$ and $\Delta t$, for an infinitesimally small $\Delta t$. Then it is still in the system at time $t+s_{k}$, but not anymore at $t+s_{k+1}$ with probability $f_k({t})-
f_{k+1}(t)$, where $f_k(t):=
e^{-\mu (t+s_k)}.$ As a consequence, we obtain the following relation:
\begin{eqnarray*}\lefteqn{\hspace{-5mm}\Psi_j({t},\vec{z})=\lambda_j\Delta t \, b(t,\vec{z})\, \Psi_j({t}-\Delta t,\vec{z})}\\
&&+\:\sum_{i\not=j}q_{ji}\Delta t\,\Psi_i({t}-\Delta t,\vec{z})+
\left(1-\lambda_j\Delta t-\sum_{i\not =j}q_{ji}\Delta t\right)\Psi_j ({t}-\Delta t,\vec{z})+o(\Delta t),\end{eqnarray*}
where
\begin{eqnarray*}
b(t,\vec{z})&:=& 
(1-f_1(t))+z_1(f_1(t)-f_2(t))+\cdots\\
&&+\,(z_1\cdots z_{K-1})(f_{K-1}(t)-f_{K}(t))+
(z_1\cdots z_{K})f_K(t).
\end{eqnarray*}
With elementary manipulations, we obtain
\[\frac{\Psi_j(t,\vec{z})-\Psi_j(t-\Delta t,\vec{z})}{\Delta t}=\sum_{i=1}^d q_{ji} \Psi_i(t-\Delta t,\vec{z})+a_j(t,\vec{z}) \Psi_j(t-\Delta t,\vec{z})+o(1),
\]
where
$a_j(t,\vec{z}):=\lambda_j  \left(
b(t,\vz) -1\right).$
Now letting $\Delta t\downarrow 0$, and defining  $A(t,\vec{z}):={\rm diag}\{\vec{a}(t,\vec{z})\}$, we obtain the differential equation, in vector notation,
\[ \frac{\partial }{\partial t}\vec{\Psi}(t,\vec{z}) =
(Q+A(t,\vec{z})) \vec{\Psi}(t,\vec{z}) .\]

\subsection{Covariance} We now explicitly compute $\COV(M(s),M(t))$, assuming, without loss of generality, that $s\le t$; the computations are similar to the ones
in Section \ref{expli} (and therefore some steps are left out).
The `law of total covariance', with $J\equiv (J( r))_{r=0}^t$, entails that
\begin{equation}\label{TCov}\COV(M(s),M(t)) = \MEAN\, \COV(M(s),M(t)\,|\,J)+\COV(\MEAN(M(s)\,|\,J),\MEAN(M(t)\,|\,J)).\end{equation}
Due to the fact that $M(s)$ obeys a Poisson distribution
with the random parameter 
$\varphi(J)$, the second term in the right hand side of (\ref{TCov}) can be written as $I_1+I_2$, where
\begin{eqnarray*}
I_1&:=&\sum_{i=1}^d\sum_{j=1}^d\lambda_i\lambda_jK_{ij},\:\:\mbox{where}\:\:K_{ij}:=\int_0^s\int_0^v e^{-\mu(s-u)} e^{-\mu(t-v) }\pi_i\left(p_{ij}(v-u)-\pi_j\right){\rm d}u{\rm d}v,\\
I_2&:=&\sum_{i=1}^d\sum_{j=1}^d\lambda_i\lambda_jL_{ij},\:\:\mbox{where}\:\:L_{ij}:=\int_0^s\int_v^t e^{-\mu(s-u)} e^{-\mu(t-v) }\pi_j\left(p_{ji}(u-v)-\pi_i\right){\rm d}u{\rm d}v.
\end{eqnarray*}
It takes some standard algebra
to obtain \begin{eqnarray*}K_{ij}&=& e^{-\mu t }\pi_i
\int_0^s\left(\int_w^s e^{2\mu v}{\rm d}v\right) e^{-\mu(s+w)} \left(p_{ij}(w)-\pi_j\right){\rm d}w\\
& =&\frac{1}{2\mu} e^{-\mu t} \pi_i \int_0^s \left(e^{\mu(s-w)}-e^{-\mu(s-w)}\right)\left(p_{ij}(w)-\pi_j\right){\rm d}w.\end{eqnarray*}
Similarly,  $L_{ij}=L_{ij}^{(1)}+L_{ij}^{(2)}$, where
\begin{eqnarray*}
L_{ij}^{(1)}&:=&\frac{1}{2\mu}e^{-\mu t}\pi_j 
\left(e^{\mu s}-e^{-\mu s}\right)
\int_0^{t-s} e^{\mu w}
 \left(p_{ji}(w)-\pi_i\right){\rm d}w
,\\
L_{ij}^{(2)}&:=&\frac{1}{2\mu}e^{-\mu s}\pi_j \int_{t-s}^t
\left(e^{\mu( t -w)}- e^{-\mu(t-w)}\right) \left(p_{ji}(w)-\pi_i\right){\rm d}w.
\end{eqnarray*}
Now concentrate on the first term in the right hand side of (\ref{TCov}). To this end, consider the following decomposition:
\[M(s):= M^{(1)}(s,t)+ M^{(2)}(s,t),\:\: \:\:M(t):= M^{(2)}(s,t)+ M^{(3)}(s,t),\]
where $M^{(1)}(s,t)$ are the jobs that arrived in $[0,s)$ that are still present at time $s$ 
but have left at time $t$, $M^{(2)}(s,t)$ the jobs that have arrived in 
$[0,s)$ that are still present at time $t$, and $M^{(3)}(s,t)$ the jobs that have arrived in $[s,t)$ that are still present at time $t$.
 Observe that, conditional on $J$,
these three random quantities are independent. As a result,
\[\MEAN\,\COV(M(s),M(t)\,|\,J) = \MEAN\, \VAR (M^{(2)}(s,t)\,|\,J).\]
Mimicking the arguments used in \cite{DAURIA}, it is immediate that 
$M^{(2)}(s,t)$ has a Poisson distribution with random parameter $\xi(J)$, where
\[\xi(f):=\int_0^s \lambda_{f( r)} e^{-\mu_{f(r )}(t-r)}{\rm d}r.\]
We conclude that
\[
\MEAN\,\COV(M(s),M(t)\,|\,J)=\MEAN\xi(J) =\sum_{i=1}^d \pi_i\lambda_i\int_0^s e^{-\mu(t-r)}{\rm d}r = \varrho(s)\,e^{-\mu(t-s)}.\]
When scaling ${\boldsymbol\lambda} \mapsto N {\boldsymbol\lambda}$ and $Q\mapsto N^\alpha
Q$, for $\alpha>0$,
it is readily verified that for $N$ large,
\[
\COV(M\hN(s),M\hN(t))\sim N \varrho(s)\,e^{-\mu(t-s)}+
N^{2-\alpha}\frac{e^{-\mu t}}{\mu} \left(e^{\mu s} -e ^{-\mu s}\right)U,\]
recalling that $U:=\m\pi^{\rm T}\Lambda D\Lambda\m 1.$
When taking $s=t$, we obtain formulae for the variance that are in line with our findings of Section \ref{expli}.

\subsection{Limit results} We again consider the situation in which
the modulating Markov chain $J(\cdot)$ is sped up by a factor $N^\alpha$ (for some {positive} $\alpha$), while the arrival rates $\lambda_i$ are sped up by $N$. In this subsection we consider the (multivariate) distribution of
the number of jobs in the system at different points in time. While in \cite{BKMT} we just covered the case of $\alpha>1$, we now establish a {\sc clt} for general $\alpha$.

As the techniques used are precisely the same as before, we just state the result. We first introduce some notation.
Define $[\check C(t)]_{k\ell}=[\check C(t)]_{\ell k},$ where for $k\ge \ell$
\[[\check C(t)]_{k\ell}:=\frac{U}{\mu}\left(1-e^{-2\mu(t+s_\ell)}\right)e^{-\mu(s_k-s_\ell)}1_{\{\alpha \le 1\}}
+\frac{\lambda_\infty}{\mu}\left(1-e^{-\mu(t+s_\ell)}\right)e^{-\mu(s_k-s_\ell)}1_{\{\alpha \ge 1\}}.\]

\begin{thm}\label{THMMult}
The random vector
\[ \left( \frac{{M}\hN(t+s_1) - N\varrho(t+s_1)}{{N^\gamma}},\ldots,
\frac{{M}\hN(t+s_K) -N {\varrho}(t+s_K)}{N^\gamma}\right) \]
converges to a $K$-dimensional Normal distribution with zero mean and covariance matrix $\check C(t)$ as $N\to\infty$.
The parameter $\gamma$ equals $\max\{1-\alpha/2,1/2\}$.
\end{thm}

As $t\to\infty$, $\check C(t)\to \check C$, where
\[[\check C]_{k\ell}=\frac{u_{k\ell}}{2\mu},\:\:\:\mbox{with}\:\:\:
u_{k\ell}:=
2\left(U 1_{\{\alpha \le 1\}}+\lambda_\infty 1_{\{\alpha \ge 1\}}\right) e^{-\mu(s_k-s_\ell)}.\]
We observe that the limiting centered and scaled process, as $t\to\infty$, has the correlation structure of an Ornstein-Uhlenbeck process
$S(t)$
(at the level of finite-dimensional distributions), that is, the solution to the stochastic differential equation
\[{\rm d}S(t) = - \mu S(t) {{\rm d}t}+\left(2U 1_{\{\alpha \le 1\}}+\sqrt{\lambda_\infty +\mu\varrho(t)}1_{\{\alpha \ge 1\}}\right){\rm d}W(t),\]
with $W(\cdot)$ standard Brownian motion.

\section{Numerical illustration}
\label{sec:illustration}

In this section, we briefly illustrate the accuracy of the approximations that are suggested by the limit theorems of this paper.
In particular, we consider the variance of the queue content $M\hN$
under stationarity for Model {\sc i}. In this case, there is an exact expression for the variance \cite{OCINNEIDEPURDUE}:
\[ \operatorname{\mathbb{V}ar}[M\hN] = 2 N^2 \m\pi^{\rm T} \Lambda (\M - N^{\alpha} Q)^{-1} \Lambda (2 \M - N^{\alpha} Q)^{-1} \m 1 + N \varrho\hI - N^2 (\varrho\hI)^2. \]

On the other hand, Theorem 2 suggests the following asymptotic expression for the variance of $M\hN$:
\[ V_1(N) := N \varrho\hI 1_{\{\alpha\geq1\}} + N^{2-\alpha} \sigma_m^2 1_{\{\alpha\leq1\}}. \]

This expression discards one of the two contributions to the variance, and may therefore be less accurate
when both terms are of comparable size.  To remedy this effect,
we propose the following simple alternative
\[ V_2(N) := N \varrho\hI + N^{2-\alpha} \sigma_m^2,\]
which is asymptotically equivalent with $V_1(N)$ as $N$ grows large.

In Fig.~\ref{fig:loglog}, we illustrate these approximations in the three different regimes
$\vr \gg \sigma_m^2$, $\vr \approx \sigma_m^2$, and $\vr \ll \sigma_m^2$, for
a two-state Markov process with generator $Q$ and varying values of $\alpha$. 
The parameter values for the three cases are
\[ Q = \begin{pmatrix} -1 & 1 \\ 3 & -3 \end{pmatrix},\quad 
\begin{pmatrix} -2 & 2 \\ 1 & -1 \end{pmatrix}, \quad 
\begin{pmatrix} -1 & 1 \\ 3 & -3 \end{pmatrix},
\]
and $\m \lambda = [1,2],\;[1,2],\;[1,50]$, $\m\mu = [2,1],\;[100,1],\;[2,1]$.
We observe that in all cases both approximations tend to the exact values as $N$
gets larger, but the errors are dependent on the specific choices of the
parameters of the Markov process.
As to be expected, $V_2(N)$ is the more accurate one.
The contourplots in the middle row give the relative error in the approximation $V_1(N)$. 
They nicely show the effect of the
absence of one of the terms in the approximation: for $\vr \gg \sigma_m^2$ the relative error is almost
one if $\alpha=1-\varepsilon$, wheras for $\vr \ll \sigma_m^2$ this is the case for $\alpha=1+\varepsilon$.
If the two terms are in balance ($\vr \approx \sigma_m^2$), we see an increase of the relative error around
$\alpha \approx 1$, which is absent  in approximation $V_2(N)$, plotted in the bottom row.

\begin{figure}
\includegraphics[width=0.9\textwidth]{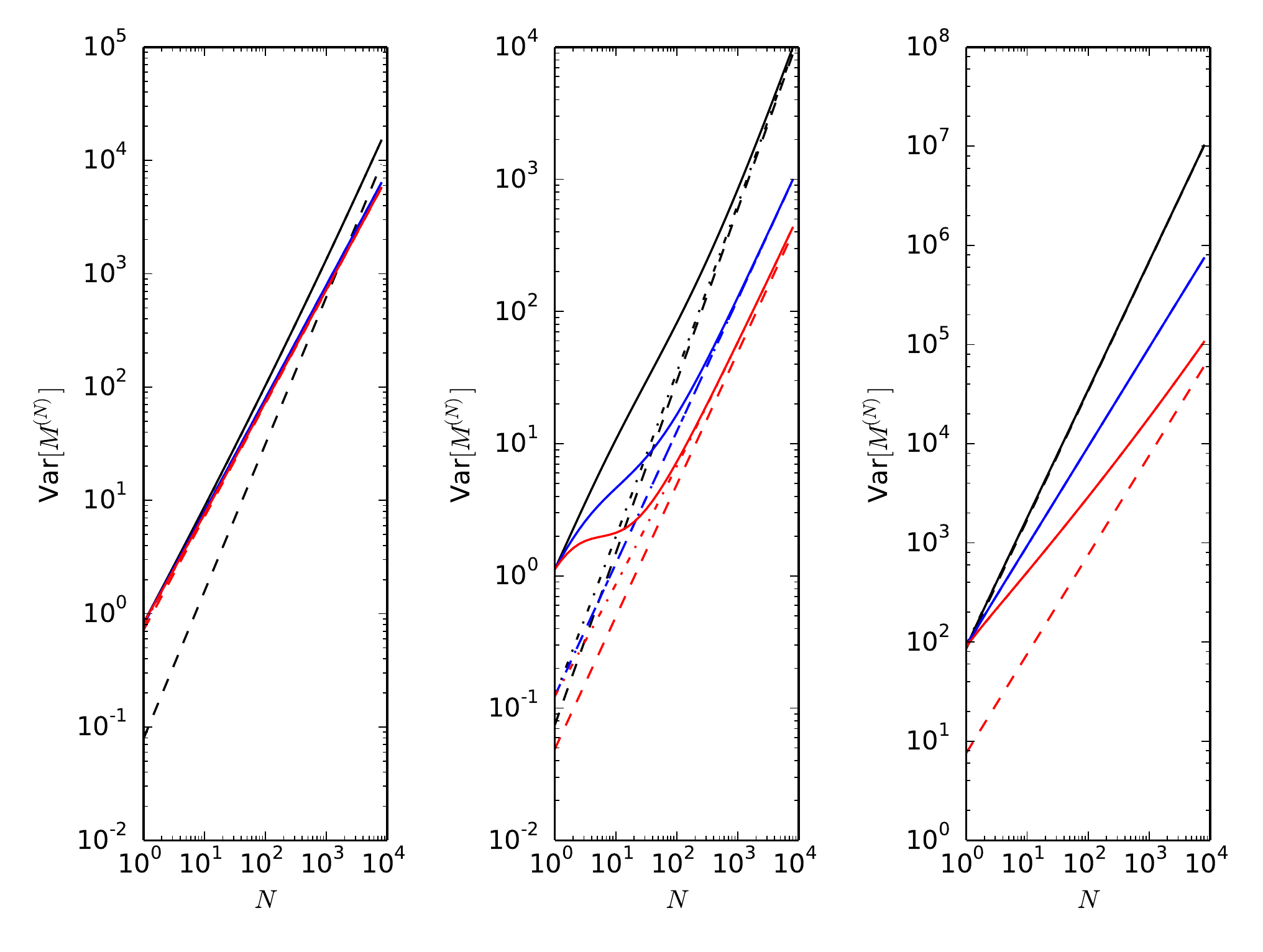}

\vspace{-0.5cm}
\includegraphics[width=\textwidth]{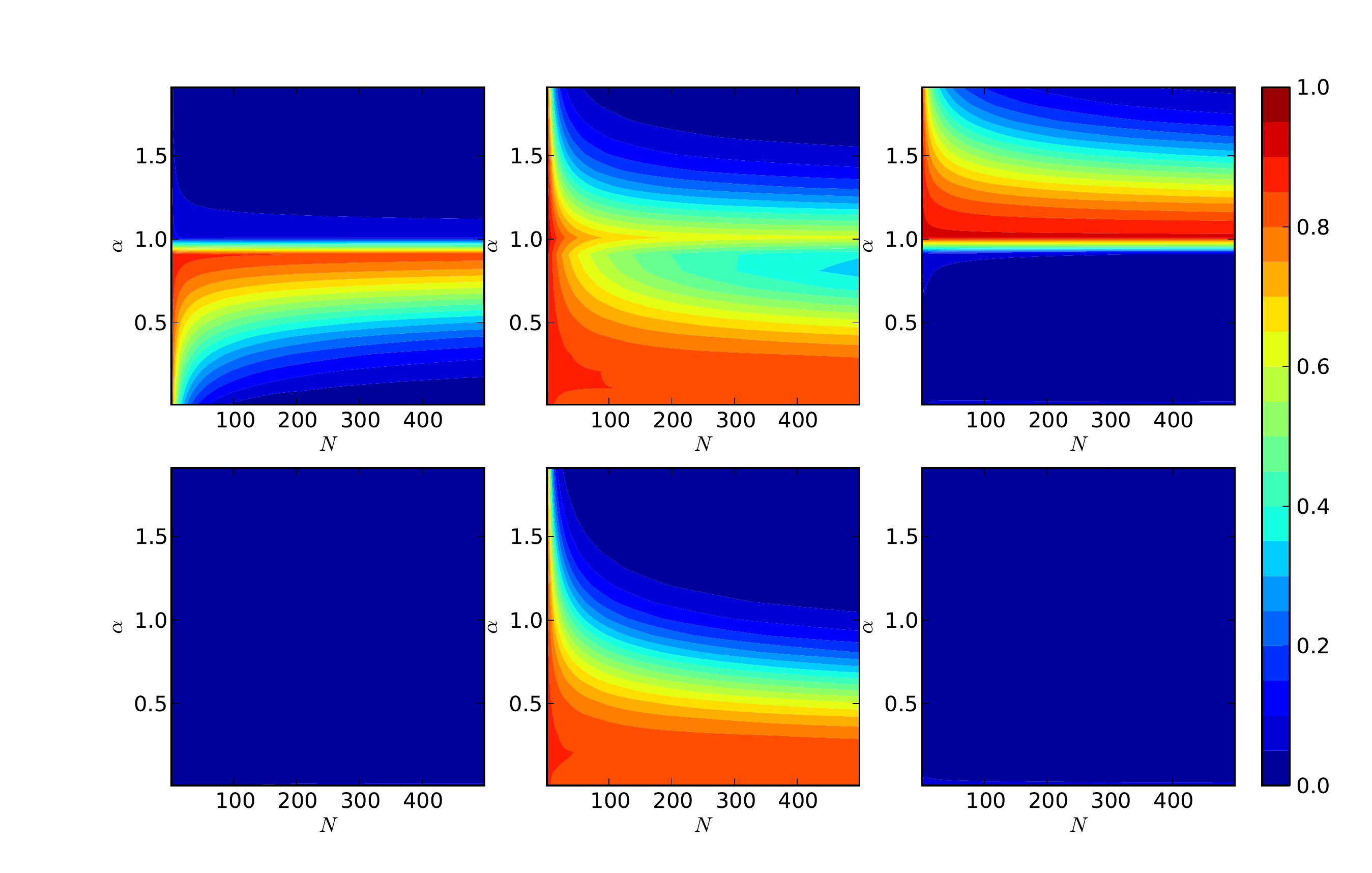}
\caption{Illustration of the behaviour of the approximation in the three different regimes: 
(left) $\vr \gg \sigma_m^2$, (middle) $\vr \approx \sigma_m^2$, (right) $\vr \ll \sigma_m^2$.
Top row: Plots of the variance of $M^{(N)}$ along with two approximations; 
black: $\alpha=0.7$; blue: $\alpha=1.0$; red: $\alpha=1.3$.
Full lines represent the exact values, dashed lines represent the first
approximation $V_1(N)$ and dash-dotted lines represent $V_2(N)$.
Middle row: Contourplots of the relative error in the approximation $V_1(N)$ for varying $\alpha$ and $N$.
Bottom row: Same for approximation $V_2(N)$.}
\label{fig:loglog}
\end{figure}

\section{Discussion and conclusion} 
\label{sec:discussion}
In this paper we derived central limit theorems ({\sc clt}\,s) for infinite-server queues
with Markov-modulated input. In our approach the modulating Markov chain is sped up by a
factor $N^\alpha$ (for some positive $\alpha$), while the arrival process is sped up by 
$N$. Interestingly, there is a {\it phase transition} in the sense that the 
normalization to be 
used in the {\sc clt} depends on the value of $\alpha$: rather than the standard 
normalization by $\sqrt{N}$, it turned out that the centered process should be divided 
by $N^\gamma$, with $\gamma$ equal to $\max\{1-\alpha/2,1/2\}.$ We have proved this by first establishing systems of differential equations for the (transient and stationary) distribution of the number of jobs in the system, and then studying their behavior under the scaling described above.

We have also derived a {\sc clt} for the {\it multivariate} distribution of the number of jobs present at different time instants, complementing the analysis for just $\alpha>1$ in \cite{BKMT}. We anticipate weak convergence to an Ornstein-Uhlenbeck process with appropriate parameters, but establishing such a claim  will require different techniques.


\appendix 
\section{Uniqueness of solutions of the PDEs}
\label{sec:uniqueness}
{
In the various proofs of this article, we have `solved' the differential
equations by guessing a solution and establishing that it satisfies both the
differential equation itself and the boundary conditions. We now show that the
solutions are indeed unique by relying on the method of characteristics
\cite{COURANT}. The method consists of rewriting the partial differential
equation ({\sc pde}) as a system of ordinary differential equations along so-called
characteristic curves, for which the theory of existence and uniqueness is
well-developed.}

{
As all occurring {\sc pde}\,s are of a similar form
and moreover quasi-linear, we can suffice by establishing uniqueness for
the two types of {\sc pde}\,s, the first of which is as follows:
\[ \sum_{k=1}^d \mu_k{\vartheta}_k \,\pd\phi{{\vartheta}_k} = g(\vec{\vartheta})\, \phi({\vartheta}_1,\ldots,{\vartheta}_d), \]
for some function $g(\cdot)$ with boundary condition $\phi(0,\ldots,0) = 1$.
This pertains to differential equations in the proofs of Lemma \ref{LEM4} 
and Thm.~\ref{THMM2}.
Let us consider a parametric curve
\[ \left({\vartheta}_1(t),\cdots,{\vartheta}_d(t),\phi(t)\right),
\]
where $\phi(t) := \phi({\vartheta}_1(t), \cdots, {\vartheta}_d(t))$ (with a slight but customary abuse of notation), subject to the following system of ordinary differential equations ({\sc ode}\,s):
\[ \od{{\vartheta}_k(t)}t = \mu_k {\vartheta}_k(t) \quad\quad\text{ and }\quad\quad \od{\phi(t)}t = g({\vartheta}_1(t),\ldots,{\vartheta}_d(t)) \phi(t). \]}

{
The {\sc ode}\,s in ${\vartheta}_k(t)$ have the following solution:
\[ {\vartheta}_k(t) = {\vartheta}_k(0) \exp(\mu_k t), \]
while the {\sc ode} for $\phi$ is also quasi-linear with a continuous function
$g(\cdot)$, such that a general solution can be found with one undetermined
constant. In order to construct the solution at an arbitrary point
$({\vartheta}_1,\ldots,{\vartheta}_d)$, one puts ${\vartheta}_k(0)={\vartheta}_k$ and then combines
this with the boundary condition $1=\phi(0,\cdots,0)$, which indeed gives us the condition to make the solution of the {\sc ode} in $\phi(t)$ unique.}

{
Next, we consider the {\sc pde}:
\[ \pd\phi t + \sum_{k=1}^d \mu_k{\vartheta}_k \pd\phi{{\vartheta}_k} = g(t,\mathbf{\vartheta})\, \phi(t, {\vartheta}_1,\ldots,{\vartheta}_d), \]
with the boundary condition $\phi(0,{\vartheta}_1,\ldots,{\vartheta}_d)=1$ (i.e., an empty
system at $t=0$) for which the uniqueness question can be tackled in a similar
but slightly different fashion (as $t$ is now an explicit variable of the
problem). This form occurs in the proofs of Thms.\ \ref{TH1} and \ref{THMM2}  (as well
as in the proofs Lemmas \ref{LEM1} and \ref{LEM4} with the slight difference that there is a
negative sign in the $\partial/\partial t$-term, which hardly changes our argument).
Indeed, we consider the  parametric curve:
\[ \left(t,{\vartheta}_1(t),\cdots,{\vartheta}_d(t),\phi(t)\right),
\]
with the same {\sc ode}\,s imposed on ${\vartheta}_k(t)$ (and hence having the same solution as well), while
\[  \od{\phi(t)}t = g(t, {\vartheta}_1(t),\ldots,{\vartheta}_d(t)) \,\phi(t) \]
has again a solution with one undetermined constant.
In order to find the solution at $(t,{\vartheta}_1,\ldots,{\vartheta}_d)$, we put ${\vartheta}_k(t)={\vartheta}_k$,
from which we find ${\vartheta}_k(0) = {\vartheta}_k \exp(-\mu_k t)$. These relations together
with $\phi(0)=1$ ensure that each {\sc ode} has a unique solution, and hence the original {\sc pde} has a
unique solution as well.}


{\small

\end{document}